\documentclass[10pt, reqno]{amsart}
\usepackage{amsmath,amssymb,amsthm,amscd}
\usepackage{enumerate}
\usepackage{tikz-cd}

\usepackage[margin=1.4in]{geometry}

\usepackage{mathrsfs}
\usepackage{dsfont}
\usepackage[T1]{fontenc}

\usepackage{hyperref}
\usepackage{comment}

\theoremstyle{plain}
\newtheorem{theorem}{Theorem}[section]
\newtheorem{lemma}[theorem]{Lemma}
\newtheorem{proposition}[theorem]{Proposition}
\newtheorem{corollary}[theorem]{Corollary}
\newtheorem{condition}[theorem]{Condition}

\theoremstyle{definition}
\newtheorem{definition}[theorem]{Definition}
\newtheorem{notation}[theorem]{Notation}

\theoremstyle{remark}

\newtheorem{example}[theorem]{Example}
\newtheorem{remark}[theorem]{Remark}

\DeclareMathOperator{\ord}{ord}
\DeclareMathOperator{\diag}{diag}

\newcommand{\norm}[1]{\left\lVert#1\right\rVert}
\newcommand{\inner}[2]{\langle #1, #2\rangle}
\renewcommand{\tilde}{\widetilde}
\renewcommand{\bar}{\overline}

\newcommand{\A}{\mathscr{A}}
\newcommand{\C}{\mathbb{C}}
\newcommand{\N}{\mathbb{N}}
\newcommand{\Q}{\mathbb{Q}}
\newcommand{\R}{\mathbb{R}}
\renewcommand{\S}{\mathscr{S}}
\newcommand{\T}{\mathbb{T}}
\newcommand{\Z}{\mathbb{Z}}

\numberwithin{equation}{section}

\makeatletter
\@namedef{subjclassname@2020}{%
	\textup{2020} Mathematics Subject Classification}
\makeatother

	\title[Constructing equivalence bimodules between noncommutative solenoids]{Constructing equivalence bimodules between noncommutative solenoids: a
		two-pronged approach}
	
	\author[Shen Lu]{Shen Lu}
	
	\date{\today}
	\address{%
		Department of Mathematics,
		University of Colorado Boulder,
		Campus Box 395,
		Boulder, CO 80309-0395, USA}
	\email{shen.lu@colorado.edu}
\begin{document}
	\begin{abstract}
		We revisit and generalize the application of a method introduced by Latr{\'e}moli{\`e}re and Packer for constructing finitely generated projective modules over the noncommutative solenoid C*-algebras.
		By realizing them as direct limits of rotation algebras, the method constructs directed systems of equivalence bimodules between rotation algebras that satisfy the necessary compatibility conditions to build Morita equivalence bimodules between the direct limit C*-algebras.
		In the irrational case, we use a fixed projection in a matrix algebra over the rotation algebra satisfying a key condition to build an equivalence bimodule at each stage following a construction of Rieffel.
		From this, our main result shows that two irrational noncommutative solenoids are Morita equivalent if and only if such a projection exists.
		We also make additional observations about the Heisenberg bimodules construction studied by the aforementioned two authors and connect the two constructions.
	\end{abstract}
    \maketitle
    
	\section{Introduction}\label{S-introduction}
		
	In this paper, we continue the study of the noncommutative solenoids introduced by Latr{\'e}moli{\`e}re and Packer \cite{LP1}.
	For a fixed prime $p$, they are defined as the twisted group C*-algebras on $\Z\left[ \frac{1}{p}\right]$, the additive group of the ring of integers adjoining the multiplicative inverse of $p$, with the discrete topology.
	These C*-algebras are one of the first examples of twisted group C*-algebras associated to non-compactly generated abelian groups with in-depth analysis.
	A focus of Latr{\'e}moli{\`e}re and Packer's earlier work was on explicit constructions of finitely generated projective modules over the noncommutative solenoids \cite{LP2, LP3}. 
	Since then, the same two authors have also studied these C*-algebras in the context of noncommutative metric geometry \cite{LP17}.
	More recently, N. Brownlowe, M. Hawkins, and A. Sims introduced a class of Toeplitz extensions to these C*-algebras and investigated their associated Kubo-Martin-Schwinger (KMS) states \cite{BHS19}.
	Another reason to study these C*-algebras is that they are $\mathcal{Z}$-stable and fit into the Elliott classification program, but the fact that their $K$-groups are infinitely generated makes the study highly non-trivial. Yet another point of interest is that these C*-algebras can also be expressed as groupoid C*-algebras, which have recently proven to be of interest in the research program of R. Deeley, I. Putnam, and K. Strung \cite{DPS18, DPS19}, which is also centered on classification.
	
	This paper is organized as follows.
	In the preliminaries section, we recall the definition of noncommutative solenoids $\A_\alpha^\S$ as twisted group C*-algebras.
	More importantly, we focus on the alternative description that realizes $\A_\alpha^\S$ as a direct limit of rotation algebras.
	Specifically, fix prime $p$ and let $\alpha = \left( \alpha_n \right)_{n \in \N}$ be a sequence of real numbers from $[0, 1)$ such that for all $n \in \N$, $p\alpha_{n + 1} = \alpha_n + x_n$ for some $x_n \in \{ 0, \dotsc, p-1 \}$.
	Then the sequence of rotation algebras $A_{\alpha_0} \rightarrow A_{\alpha_2} \rightarrow A_{\alpha_4} \rightarrow \cdots$ converges to $\A_\alpha^\S$, where the connecting maps send the generators $U_{\alpha_{2n}}$ and $V_{\beta_{2n}}$ of $A_{\alpha_{2n}}$ to the $p$th powers of the corresponding generators of $A_{\alpha_{2n+2}}$.
	Moreover, we can naturally associate  to each $\alpha$ a $p$-adic integer that is given by $x_\alpha = \sum_{j = 0}^\infty x_j p^j$.
	We will exploit this connection to $p$-adic analysis throughout our study.
	Since $A_{\theta} \cong A_{\theta + n}$ for any $n \in \Z$, changing any entry of $\alpha$ by an integer does not affect the resulting solenoid.
	For this reason, we introduce an additive group $\Omega_p$ of real number sequences satisfying $p \alpha_{n+1} \equiv \alpha_n \mod \Z$ and define a unique noncommutative solenoid from $\alpha \in \Omega_p$ as a direct limit of rotation algebras the same way as before.
	When $\alpha$ is an irrational sequence, $\A_\alpha^\S$ is a direct limit of irrational rotation algebras, which have been studied extensively.
	Since much is known about the structure of their projective modules, we primarily concern ourselves with the irrational noncommutative solenoids.
	
	In Section \ref{S-heisenberg bimodules over NS}, we recall a construction of Latr{\'e}moli{\`e}re and Packer for forming projective modules over noncommutative solenoids.
	The construction is originally due ot Rieffel \cite{Rieffel88} and commonly referred to as the Heisenberg bimodules.
	In summary, one can embed $\Z\left[ 1/p \right] \times \Z\left[ 1/p\right]$ into the self-dual group $M = \R \times \Q_p$ as a discrete cocompact subgroup, where $\Q_p$ is the field of $p$-adic numbers.
	The embedding can be done in such a way so that the twisted group C*-algebra associated to the image in $M$ (with the Heisenberg multiplier) is isomorphic to $\A_\alpha^\S$ and the one associated to the annihilator of the image (with the conjugated Heisenberg multiplier) is isomorphic to $\A_\beta^\S$ for some $\beta \in \Omega_p$.
	In this case, $C_c\left( M \right)$, suitably completed, is an equivalence bimodule between $\A_\alpha^\S$ and $\A_\beta^\S$.
	Our contribution here is to simply give an explicit formula for $\beta$, using basic facts about the $p$-adic numbers.
	This formula will be useful when we relate our second construction to the Heisenberg bimodules in the irrational case.
	
	The main concern of this paper is the open problem of classifying the irrational noncommutative solenoids up to (strong) Morita equivalence.
	Since these C*-algebras are unital, $\A_\alpha^\S$ and $\A_\beta^\S$ are Morita equivalent if and only if one is a full corner of the algebra of $k \times k$ matrices over the other for a suitable $k$.
	That is, there exists a projection $P \in M_k \left(  \A_\alpha^\S \right)$ such that $\A_\beta^\S \cong P M_k \left(  \A_\alpha^\S \right) P$.
	It is sufficient to assume that $P \in M_k\left( A_{\alpha_{2N}} \right)$ for some $N \in \N$, as $\A_\alpha^\S = \varinjlim A_{\alpha_{2n}}$ and the connecting maps are all unital embeddings.
	For the same reason, one can regard $P$ as a projection in $M_k\left( A_{\alpha_{2n}} \right)$ for all $n \geq N$. In Section \ref{S-directed systems}, using such a projection and a formulation of Rieffel, we establish a Morita equivalence bimodule $X_{2n}$ between $A_{\alpha_{2n}}$ and $A_{\beta_{2n}}$, where $\beta_{2n}$ is the appropriate fractional linear transformation of $\alpha_{2n}$ for each $n$. 
	In \cite{LP3}, Latr{\'e}moli{\`e}re and Packer introduced a notion of directed systems of equivalence bimodules adapted to directed systems of unital C*-algebras that gives rise to an equivalence bimodule between the direct limit C*-algebras.
	In their definition, the bimodules embeddings must satisfy the necessary compatibility conditions in relation to the connecting maps in the direct limit C*-algebras.
	We define explicit embeddings of $X_{2n}$ into $X_{2n+2}$ that give us a directed system of equivalence bimodules between two sequences of irrational rotation algebras, each of which converges to a noncommutative solenoid. 
	Much of Section \ref{S-directed systems} is devoted to checking the bimodule embeddings are compatible with the connecting maps required to form a noncommutative solenoid. 
	Additionally, motivated by finding all such projections, we give a key condition on projections that is both necessary and sufficient to establish equivalence bimodules between solenoids using the directed systems construction.
	
	In Section \ref{S-morita-problem}, we address the question of Morita equivalence for irrational noncommutative solenoids. Specifically, we conclude that $\A_\alpha^\S$ and $\A_\beta^\S$ are Morita equivalent if and only if there exists a projection in a matrix algebra over $\A_{\alpha}^\S$ that satisfies our key condition.

	\begin{notation}
		Throughout this paper, the natural numbers $\N$ include zero. Unless otherwise specified, we always follow the convention that $\sum_{j = i}^{n} \cdots = 0$ if $n < i$. 
	\end{notation}


	\section{Review of the Noncommutative Solenoids}\label{S-review of nc solenoids}
	
	In this section, we recall the necessary definitions concerning the noncommutative solenoids.
	
	Fix prime $p$, and let
	$$
	\Z\left[ \frac{1}{p} \right] := \left\{ \frac{j}{p^k} \in \Q \:: j \in \Z, k \in \N \right\}
	$$
	be the additive subgroup of $\Q$ consisting of rational numbers whose denominators are nonnegative integral powers of $p$.
	Naturally, we insist that all elements of $\Z\left[ 1/p\right]$ are written in their unique reduced form, that is, the exponent of $p$ in the denominator is minimal. We endow $\Z\left[ 1/p \right]$ with the discrete topology.
	
	\begin{definition}\label{D-nc solenoid}
		Fix prime $p$. A \emph{noncommutative solenoid} is a twisted group C*-algebra of the form
		\[
		C^\ast\left( \Z\left[ \frac{1}{p} \right] \times \Z\left[ \frac{1}{p} \right], \sigma \right),
		\]
		where $\sigma$ is a multiplier of the group $\Z\left[1/p \right] \times \Z\left[ 1/p \right]$. To easy notation, we will denote by $\Gamma$ the discrete group $\Z\left[1/p \right] \times \Z\left[ 1/p \right]$.
	\end{definition}
	
	Recall that a multiplier $\sigma$ on a locally compact group $G$ is a Borel function $G \times G \rightarrow \T$, where $\T$ is the unit circle in $\C$, satisfying 
	\begin{enumerate}
		\item $\sigma(r, s) \sigma(r+s, t) = \sigma(r, s+t)\sigma(s, t)$;
		\item $\sigma(s, e) = \sigma(e, s) = 1$,
	\end{enumerate}
	for all $r, s, t \in G$, with $e$ being the identity of $G$.
	In other words, a multiplier is a normalized $2$-cocycle for the trivial group action of $G$ on $\T$.
	The set of multipliers is denoted by $Z^2(G, \T)$ and two multipliers $\sigma$, $\tau$ are cohomologous, or equivalent, if there is a Borel function $\rho: G \rightarrow \T$ such that $\sigma(s, t) = \rho(s)\rho(t)\rho(s+t)^{-1}\tau(s ,t)$ for all $s, t \in G$.
	The set of multipliers cohomologous to the multiplier that is identically $1$ form a normal subgroup of $Z^2(G, \T)$, denoted by $B^2(G, \T)$.
	The quotient group is denoted by $H^2(G, \T)$.
	If two multipliers $\sigma$ and $\tau$ are cohomologous, then the twisted group C*-algebras $C^\ast(G, \sigma)$ and $C^\ast(G, \tau)$ are $\ast$-isomorphic.
	It is then necessary to classify the multipliers up to their cohomology classes.
	This was first addressed in \cite{LP1} and we record it here, as it is crucial to our computations in Section \ref{S-heisenberg bimodules over NS}.
	
	\begin{theorem}\label{Xi_p_Psi_alpha}\emph{\cite[Theorem 2.3]{LP1}}
		Fix prime $p$, and the set
		\[
		\Xi_p := \{ (\alpha_n)_{n \in \N} \:: \alpha_0 \in [0, 1) \text{ and } \forall n \in \N, \exists x_n \in \{ 0, \dotsc, p-1 \} \text{ such that } p \alpha_{n+1} = \alpha_n + x_n  \}
		\]
		is a group under pointwise addition modulo one.
		There exists a group isomorphism $\rho: H^2(\Gamma, \T) \rightarrow \Xi_p$ such that if $\sigma \in H^2(\Gamma, \T)$ and $\alpha = \rho\left( \sigma \right)$, and if $f$ is a multiplier of equivalence class $\sigma$, then $f$ is cohomologous to the multiplier:
		\[
		\Psi_\alpha: \begin{cases}
			\qquad \qquad \quad \Gamma \times \Gamma & \rightarrow \T\\
			\left( \left( \dfrac{j_1}{p^{k_1}} , \dfrac{j_2}{p^{k_2}}   \right), \left( \dfrac{j_3}{p^{k_3}} , \dfrac{j_4}{p^{k_4}} \right) \right) &\mapsto \exp \left( 2 \pi i \alpha_{(k_1+k_4)}j_1j_4  \right).
		\end{cases}
		\]
	\end{theorem}
	
	It follows immediately that $\Psi_\alpha$ and $\Psi_\beta$ are cohomologous if and only if $\alpha = \beta \in \Xi_p$.
	Notice that changing any number of entries of $\alpha \in \Xi_p$ by integer values does not change the multiplier, hence the noncommutative solenoid it defines.
	However, changing any entry of $\alpha$ by a nonzero integer takes $\alpha$ out of the group $\Xi_p$, so it is convenient to define the following group:
	\[
	\Omega_p := \{ (\alpha_n)_{n \in \N} \in \R^\N \:: \forall n \in \N, \exists x_n \in \Z \text{ such that } p \alpha_{n+1} = \alpha_n + x_n \},
	\]
	with pointwise addition.
	The map $h: \Omega_p \rightarrow \Xi_p$ given by $h\left( (\alpha_n)_{n \in \N} \right) = \left( \alpha_n \mod \Z \right)_{ n \in \N}$ defines a surjective group homomorphism.
	Given any $\alpha \in \Xi_p$, the corresponding noncommutative solenoid $\A^\S_\alpha = C^*(\Gamma, \Psi_\alpha)$ can also be defined by any element in the fiber of $\alpha$ under $h$.
	Frequently, a noncommutative solenoid will be defined by some $\alpha \in \Omega_p$, and we make no distinction between $C^\ast(\Gamma, \Psi_\alpha)$ and $C^\ast(\Gamma, \Psi_{h\left( \alpha \right)})$ for any $\alpha \in \Omega_p$.
	
	\begin{remark}\label{determine-alpha}
		It is easy to see that for any $\alpha \in \Xi_p$, $\alpha_{N}$ determines $\alpha_n$ for all $n \leq N$.
		To uniquely determine an element $\alpha = (\alpha_n)_{n \in \N}$ in $\Xi_p$, it is sufficient to know infinitely many entries (for example, $\alpha_{2n}$ for all $n \in \N$).
		This is not the case for $\Omega_p$, as many different elements of $\Omega_p$ could agree at infinitely many entries.
		However, it is straightforward to check that if $\alpha$ and $\tilde{\alpha}$, both in $\Omega_p$ and agree at infinitely many entries, then $h(\alpha) = h(\tilde{\alpha})$ and $C^\ast(\Gamma, \Psi_\alpha) \cong C^\ast(\Gamma, \Psi_{\tilde{\alpha}})$.
	\end{remark}
	
	Recall that for any $\theta \in \R$, the rotation algebra $A_\theta$ is the universal C*-algebra generated by two unitaries $U_\theta$ and $V_\theta$ satisfying the relation $U_\theta V_\theta = e^{2 \pi i \theta} V_\theta U_\theta$.
	The following alternative characterization of the noncommutative solenoids as direct limits of rotation algebras will be crucial to us in Section \ref{S-directed systems}.
	Specifically, it allows us to build equivalence bimodules over irrational noncommutative solenoids from equivalence bimodules over irrational rotation algebras.
	
	\begin{theorem}\emph{\cite[Theorem 3.7]{LP1}}\label{direct-limit-realization}
		Let $\alpha = \left( \alpha_n \right)_{n=0}^\infty \in \Xi_p$, and for each $n \in \N$, let $\varphi_n : A_{\alpha_{2n}} \rightarrow A_{\alpha_{2n+2}}$ be the unique homomorphism given by 
		\[
		\varphi_{n}(U_{\alpha_{2n}}) = U_{\alpha_{2n+2}}^p \quad \text{and} \quad \varphi_{n}(V_{\alpha_{2n}}) = V_{\alpha_{2n+2}}^p.
		\]
		The noncommutative solenoid $\A_\alpha^\S$ is the direct limit $\varinjlim(A_{\alpha_{2n}}, \varphi_{n})$. 
	\end{theorem}
	
	Since the rotation algebras $A_\theta$ and $A_{\theta + n}$ are isomorphic for any $n \in \Z$, we can again replace $\alpha \in \Xi_p$ with any $\tilde{\alpha} \in \Omega_p$ satisfying $h\left( \tilde{\alpha} \right) = \alpha$ without changing the direct limit defined from $\alpha$.
	When there is no confusion, we will take this alternative characterization as definition for the noncommutative solenoid.
	
	\begin{remark}
		It is shown in \cite[Proposition 3.3]{LP1} that $\A_\alpha^\S$ can also be written as $C(\S_p)\rtimes_{{\theta^\alpha}} \Z[1/p]$, where $\S_p$ is the $p$-solenoid group and the action of $\Z\left[ 1/p \right]$ on $\S_p$ is given by 
		\[
		\theta^\alpha_{\frac{j}{p^k}} \left( (z_n)_{n \in \N} \right) = \left( \exp\left( 2 \pi i \alpha_{k+n}j \right)z_n \right)_{n \in \N}.
		\]
		Hence it is a a groupoid C*-algebra corresponding to the associated transformation C*-algebra.
	\end{remark}
	
	In addition to characterizing each embedding by the generators, it will be necessary to describe it in greater detail, especially when each $A_{\alpha_{2n}}$ is identified as the crossed product C*-algebra $C\left( \T \right) \rtimes_{\alpha_{2n}} \Z$.
	For details of the following, see \cite[page 68]{Williams07}.
	
	Since $A_{\alpha_{2n}}$ is the completion of the $\ast$-algebra $C_c\left( \Z \times \T \right)$ with the appropriate convolution product and involution, it suffices for our purposes to extend the embedding $\varphi_{n} : A_{\alpha_{2n}} \rightarrow A_{\alpha_{2n + 2}}$ to the dense subset of $C_c\left( \Z \times \T \right)$ spanned by
	\[
	\{ f \delta_j \:: f \in C\left( \T \right) \text{ and } j \in \Z \},
	\]
	where
	\[f\delta_j (t, k) = \begin{cases}
		f(t) &\text{if } k = j,\\
		0 &\text{otherwise}.
	\end{cases}
	\] 
	\begin{lemma}\label{rotation-algebra-embedding-formula}
		With notation as above, for any $f \in C\left( \T \right)$ and $j \in \Z$,
		\[
		\varphi_{n}\left( f \delta_j\right) \left( t, k \right) = 
		\begin{cases}
			f(pt) & \text{ if } k = jp\\
			0 & \text{ otherwise.}
		\end{cases}
		\]
	\end{lemma}
	
	\begin{proof}
		In this realization, the generators $U_{\alpha_{2n}}$ and $V_{\alpha_{2n}}$ are given by $1_{C(\T)} \delta_1$ and $\iota_{\T}\delta_0$, respectively, where $1_{C(\T)}(t) = 1$ and $\iota_{\T}(t) = e^{2 \pi i t}$ for all $t \in \R$.
		Following the proof of \cite[Proposition 2.56]{Williams07}, we denote by $i_{C(\T)}$ the embedding of $C(\T)$ into $C_c\left( \T \times \Z \right)$ given by $f \mapsto f \delta_0$, then $f \delta_j = i_{C(\T)}(f) \ast U_{\alpha_{2n}}^j$ for any $j \in \Z$. 
		Since $\{ V_{\alpha_{2n}}^m : m \in \Z \}$ spans a dense subalgebra of $i_{C(\T)}\left( C(\T) \right)$, and for any $(t, k) \in \R \times \Z$, 
		\begin{align*}
			\varphi_{n}\left( {V_{\alpha_{2n}}^m} \right) (t, k) = V_{\alpha_{2n+2}}^{mp}(t, k) = V_{\alpha_{2n+2}}^m \left( pt, k \right),
		\end{align*}
		it follows that $\varphi_{n}\left( i_{C(\T)}\left( f \right) \right)(t) = f\left( pt \right) \delta_0$ for any $f \in C\left( \T \right)$.
		We compute that for any $(t, k) \in \R \times \Z$, 
		\begin{align*}
			\varphi_{n}\left( f \delta_j \right)\left( t, k\right) &= \varphi_{n}\left( i_{C(\T)}(f) \ast U_{\alpha_{2n}}^j \right) (t, k)\\
			&= \left( \varphi_{n}\left( i_{C(\T)} \left( f \right) \right) \ast \varphi_{n}\left( U_{\alpha_{2n}}^j \right) \right) (t, k) \\
			&= \left( \varphi_{n}\left( i_{C(\T)} \left( f \right) \right) \ast U_{\alpha_{2n+2}}^{pj} \right) (t, k) \\
			&= \left( \varphi_{n}\left( i_{C(\T)} \left( f \right) \right) \ast \delta_{pj} \right) (t,k) \\
			&= f \delta_{pj}\left( pt, k \right)
		\end{align*}
	\end{proof}
	
	Extending by linearity, we have that for any \textit{finite} sum $\sum_{m \in \Z} f_m \delta_m$, 
	\[
	\varphi_{n}\left( \sum_{m \in \Z} f_m \delta_m \right)(t, k) = \sum_{m \in \Z} f_m \delta_{pm}(pt, k).
	\]
	
	
	\section{Forming Heisenberg Bimodules over Noncommutative Solenoids}\label{S-heisenberg bimodules over NS}
	
	\subsection{The field of $p$-adic numbers $\Q_p$}\label{SS-p-adic numbers}
	
	We first take the standard algebraic approach to define the field of $p$-adic numbers and refer to Chapter 1 of \cite{Robert00} for a more extensive exposition.
	For a fixed prime $p$, recall that a \emph{$p$-adic integer} is a formal series $a = \sum_{j = 0}^\infty a_j p^j$ with integral coefficients $a_j$ satisfying $0 \leq a_j \leq p-1$.
	Under the usual addition and multiplication, the ring of $p$-adic integers form an integral domain with additive identity $0 = \sum_{j = 0}^\infty 0 \cdot p^j$, and multiplicative identity $1 = 1 + \sum_{j = 1}^\infty 0 \cdot p^j$.
	If $a \neq 0$, then there exists a unique natural number $v$ such that $a_v \neq 0$ and $a_j = 0$ for all $j < v$.
	We call $v$ the \emph{order of $a$} and denote it by $\ord(a)$, with the usual convention that $\ord(0) = \infty$.

	\begin{lemma}
		The group $\Z_p^\times$ of invertible $p$-adic integers consists exactly of the $p$-adic integers of order $0$. That is, $a \in \Z_p^\times$ if and only if $a_0 \neq 0$.
	\end{lemma}
    \begin{proof}
    	This is well known.
    \end{proof}
	
	The field of \emph{$p$-adic numbers} is then defined to be the field of fractions of $\Z_p$.
	Each nonzero $p$-adic number $a$ can be uniquely written as  $a = \sum_{j = v}^\infty a_j p^j$ with $v < \infty$ possibly negative, $a_v \neq 0$, and $a_j \in \left\{ 0, 1, \dotsc, p-1 \right\}$ for all $j$.
	The \emph{fractional part} of a $p$-adic number is given by
	\[
	\left\{ a \right\}_p = \sum_{j = v}^{-1} a_j p^j \in \Z\left[ 1/p \right].
	\]
	Naturally, $\left\{ a \right\}_p = 0$ if and only if $a$ is a $p$-adic integer.
	Extending the definition of order for $p$-adic integers, the order of $a$ is given by $\ord(a) = v$. For each $a \in \Q_p$, if $\ord\left( a \right) = v$, then it is easy to deduce that $\ord\left( a^{-1} \right) = -v$.
	
	Alternatively, taking an analytic approach, $\Q_p$ is the completion of $\Q$ with respect to the \emph{$p$-adic absolute value} $| \cdot |_p$: for $a = p^k \left( m/n \right) \in \Q$ with $m$ and $n$ both indivisible by $p$ (then $k$ is necessarily unique), $|a|_p = p^{-k}$.
	We note that the inclusion of $\Q$ into $\Q_p$ is an injective ring homomorphism with image exactly the set of $p$-adic numbers with coefficients that are eventually periodic.
	Moreover, $p$-adic arithmetic extends the ordinary arithmetic of the rationals.
	
	\begin{remark}\label{remark-identify-rationals}
		Unless otherwise specified, we identify a rational number with its image in the $p$-adic numbers. This also means that when convenient, we identify elements of $\Z\left[ 1/p \right]$ with the corresponding $p$-adic numbers.
	\end{remark}
	
	The following lemmas will be useful. 
	
	\begin{lemma}\label{fractional-part-congruence} For $x = \sum_{j=v}^{\infty} x_j p^j \in \Q_p$, with $v = \ord(x) < \infty$,
		\[
		\left\{ x s_1 s_2 \right\}_p \equiv \left( \sum_{j = v}^{k_1 + k_2 - 1} x_j p^j \right) \cdot s_1 s_2 \mod \Z,
		\]
		where $s_i = \frac{j_i}{p^{k_i}} \in \Z\left[ \frac{1}{p} \right]$ for $i = 1, 2$. 
	\end{lemma}
	
	\begin{proof}
		This is trivial if $v \geq k_1 + k_2$, in which case $x \in \Z_p$ and $\left\{ x s_1 s_2 \right\}_p = \left( \sum_{j = v}^{k_1 + k_2 - 1} x_j p^j \right) \cdot s_1 s_2 = 0$.
		For $v < k_1 + k_2$, observe that 
		\[
		\left\{ x s_1 s_2 \right\}_p = \left\{ \left( \sum_{j = v}^{\infty} x_j p^{j - \left( k_1 + k_2 \right)} \right) j_1 j_2  \right\}_p 
		\] 
		and
		\[
		\left( \sum_{j = v}^{k_1 + k_2 - 1} x_j p^j \right) \cdot s_1 s_2 = \left( \sum_{j = v}^{k_1 + k_2 - 1} x_j p^{j - \left( k_1 + k_2 \right)}  \right) j_1 j_2.
		\]
		It is now visible that $\left\{ x s_1 s_2 \right\}_p \equiv \left( \sum_{j = 0}^{k_1 + k_2 - 1} x_j p^j \right) \cdot s_1 s_2 \mod \Z.$
	\end{proof}
    
    \begin{lemma}\label{inverse-multiply-one}
    	Let $x = \sum_{j = v}^{\infty} x_j p^j$ be a $p$-adic integer with $0 \leq \ord(x) = v < \infty$ and inverse $x^{-1} = \sum_{j = -v}^{\infty} y_j p^j$.
    	Then for all nonnegative integer $k$,
    	\[
    	\left(  \sum_{j = -v}^{-v+k} y_j p^{j+v}  \right)\left(  \sum_{j = v}^{v+k} x_{j} p^{j - v} \right) \equiv 1 \mod p^{k + 1}.
    	\]
    \end{lemma}
    
    \begin{proof}
    	This is equivalent to
    	\[
    	\left(  \sum_{ j = 0}^{N} \tilde{y}_j p^j  \right)\left(  \sum_{ j = 0}^{N} \tilde{x}_j p^j \right) \equiv 1 \mod p^{N+1},
    	\]
    	for all nonnegative integer $N$, which holds for any invertible $p$-adic integer $x =  \sum_{j = 0}^{\infty} \tilde{x}_j p^j$ with inverse $x^{-1} =  \sum_{j = 0}^{\infty} \tilde{y}_j p^j$.
    \end{proof}

	
	\subsection{Heisenberg bimodules of Rieffel}\label{SS-Heisenberg bimodules}
	
	In this section, we only assume that $\alpha = \left( \alpha_n \right)_{n \in \N}$ is a sequence in $\Xi_p$ with $\alpha_0 \neq 0$.
	Using a construction of Rieffel \cite{Rieffel88} known as the \emph{Heisenberg bimodules}, explicit $\A_\alpha^\S$-$\mathcal{B}$-equivalence bimodules are constructed in \cite{LP2}, where $\mathcal{B}$ is isomorphic to a noncommutative solenoid.
	We will summarize the construction when applied to noncommutative solenoids and give a formula for $\beta \in \Omega_p$ such that $\mathcal{B} \cong \A_\beta^\S$ in Theorem \ref{beta-formula-general} at the end of the section.
	
	As before, we denote by $\Gamma$ the discrete group $\Z\left[ 1/p \right] \times \Z\left[ 1/p \right]$. Consider the group $M = \Q_p \times \R$. Since both $\Q_p$ and $\R$ are self-dual, $M$ is self-dual.
	Specifically, It is shown in \cite{LP2} that every character of $M$ is given by
	\[
	\chi_{(x, r)} : \begin{cases}
		\Q_p \times \R &\rightarrow \T\\
		\left( q, t \right) &\mapsto \chi_x\left( q \right) \chi_r\left( t \right),
	\end{cases}
	\]
	for some $(x, r) \in \Q_p \times \R$, where 
	\[
	\chi_x :  \begin{cases}
		\Q_p &\rightarrow \T \\
		q &\mapsto e^{2 \pi i \left\{ x q \right\}_p}
	\end{cases}
	\quad \text{and} \quad 
	\chi_r : \begin{cases}
	\R &\rightarrow \T \\
	t &\mapsto e^{2 \pi i r t }
	\end{cases}
	\]
	are characters of $\Q_p$ and $\R$, respectively. Again, $\{ x q \}_p$ is the fractional part of the $p$-adic number $x q$. 
	
	Following Rieffel \cite{Rieffel88}, the \textit{Heisenberg multiplier} $\eta : (M \times \hat{M}) \times (M \times \hat{M}) \rightarrow \T$ is defined to be
	\[
	\eta : \begin{cases}
	\qquad \quad [\Q_p \times \R]^2 \times [\Q_p \times \R]^2 &\rightarrow \T,\\
	\left( [(q_1, r_1), (q_2, r_2)], [(q_3, r_3), (q_4, r_4)] \right) &\mapsto e^{2 \pi i r_1 r_4} e^{2 \pi i \left\{ q_1 q_4 \right\}_p}.
	\end{cases}
	\]
	
	The \textit{symmetrized version} of $\eta$, denoted by $\rho$, is the following multiplier on $[\Q_p \times \R]^2$:
	\begin{multline*}
		\rho \left( [(q_1, r_1), (q_2, r_2)], [(q_3, r_3), (q_4, r_4)] \right)\\	
		= \eta \left( [(q_1, r_1), (q_2, r_2)], [(q_3, r_3), (q_4, r_4)] \right) \bar{\eta \left( [(q_3, r_3), (q_4, r_4)], [(q_1, r_1), (q_2, r_2)]\right)},
	\end{multline*}
	with $(q_i, r_i) \in \Q_p \times \R$ for $i = 1, 2, 3, 4$.

	Now, for each pair of $x \in \Q_p \setminus \{ 0 \}$ and $\theta \in \R \setminus \{ 0 \}$, let $\iota_{x, \theta}: \Gamma \rightarrow M \times \hat{M} \cong M \times M$ be the embedding of $\Gamma$ as a lattice (discrete cocompact subgroup) into $\left[ \Q_p \times \R \right]^2$ given by 
	\[
	\iota_{x, \theta} \left( r_1, r_2 \right) = \left[ (x \cdot r_1, \theta \cdot r_1), (r_2, r_2) \right].
	\]
	We denote the image of $\iota_{x, \theta}$ in $M \times \hat{M}$ by $D_{x, \theta}$.
	For each such embedding, the so-called \textit{annihilator of $D_{x, \theta}$} is defined as
	\[
	D_{x, \theta}^\perp := \left\{ [(q_1, s_1), (q_2, s_2)] \in M \times \hat{M} \:: \forall r_1, r_2 \in \Z\left[ \dfrac{1}{p} \right], \rho\left( \iota_{x, \theta}(r_1, r_2), [(q_1, s_1), (q_2, s_2)] \right) = 1 \right\}.
	\]
	
	Since $D_{x, \theta}$ is a lattice in $G$, so is $D_{x, \theta}^\perp$ \cite[Lemma 3.1]{Rieffel88}.
	
	Applying the result in \cite[Theorem 2.15]{Rieffel88} to $M = \Q_p \times \R$, $D_{x, \theta}$ and $D_{x, \theta}^\perp$, we have that $C_c(M)$, suitably completed, has the structure of a $C^\ast(D_{x, \theta}, \eta)$-$C^\ast(D_{x, \theta}^\perp, \overline{\eta})$-equivalence bimodule and implements the Morita equivalence between $C^\ast(D_{x, \theta}, \eta)$ and $C^\ast(D_{x, \theta}^\perp, \overline{\eta})$.
	We now state the main theorem of this section.
	
	\begin{theorem}\label{beta-formula-nonzero}
		Let $p$ be prime and let $\alpha = (\alpha_j)_{j \in \N} \in \Omega_p$ such that $p\alpha_{j+1} = \alpha_j + x_j$ with $\alpha_0 \neq 0$ and $x_0 \neq 0$.
		Then $C^\ast(D_{x, \theta}, \eta)$ is isomorphic to the noncommutative solenoid $\A_\alpha^\S$, where $x = \sum_{j=0}^{\infty} x_j p^j \in \Z_p$ and $\theta = \alpha_0$.
		Since $x_0 \neq 0$, we write (uniquely) $x^{-1} = \sum_{j = 0}^{\infty} y_j p^j \in \Z_p$.
		Via the Heisenberg equivalence bimodule of Rieffel, $\A_\alpha^\S$ is Morita equivalent to a noncommutative solenoid $\A_\beta^\S$, where $\beta = \left( \beta_n \right)_{ n \in \N}$ and 
		\[
		\beta = \left( \dfrac{1}{\theta}, \dfrac{1}{\theta p} + \frac{y_0}{p}, \dfrac{1}{\theta p^2} + \dfrac{y_0 + y_1 p}{p^2}, \dotsc, \beta_n, \dotsc \right), \quad \beta_n = \dfrac{1}{\theta p^n} + \dfrac{\sum_{j = 0}^{n-1} y_j p^j}{p^n}.
		\]
	\end{theorem}
	\begin{proof}
		The statement of this theorem gives an explicit formula for the $\beta$ in \cite[Theorem 5.6]{LP2}, with the additional assumption that $x_0 \neq 0$.
		We note that in the proof of the aforementioned theorem, one only needs to assume $\alpha_0 \neq 0$.
		This means that $\A_\alpha^\S$ is isomorphic $C^\ast\left( D_{x, \theta}, \eta \right)$ even when $\alpha$ is taken from $\Omega_p$.
		We simply need to show that $C^\ast\left( D_{x, \theta}^\perp, \bar{\eta} \right)$ is isomorphic to $\A_\beta^\S$ to conclude that the two noncommutative solenoids $\A_\alpha^\S$ and $\A_\beta^\S$ are Morita equivalent due to Rieffel's method.
		
		It is shown in \cite[Lemma 5.5]{LP2} that
		\[
		D_{x, \theta}^\perp = \left\{ \left[ (s_1, -s_1), \left( x^{-1}s_2, -\frac{s_2}{\theta} \right) \right] : s_1, s_2 \in \Z\left[ \frac{1}{p} \right] \right\}.
		\]
		For convenience, we choose to substitute $s_2$ with $-s_2$ and write 
		\[
		D_{x, \theta}^\perp = \left\{ \left[ (s_1, -s_1), \left( -x^{-1}s_2, \frac{s_2}{\theta} \right) \right] : s_1, s_2 \in \Z\left[ \frac{1}{p} \right] \right\}.
		\]
		Let $\lambda_{x, \theta}$ denote the embedding of $\Z\left[\frac{1}{p} \right] \times \Z\left[ \frac{1}{p} \right]$ into $\left[ \Q_p \times \R \right]^2$, so that the image of $\lambda_{x, \theta}$ is $D_{x, \theta}^\perp$. We then have 
		
		\begin{align*}
		\bar{\eta} \left( \lambda_{x, \theta}(s_1,s_2), \lambda_{x, \theta}(s_3, s_4) \right) &= \bar{\eta} \left( \left[(s_1, -s_1), \left( -x^{-1}s_2, \frac{s_2}{\theta} \right)\right], \left[(s_3, -s_3), \left( -x^{-1}s_4, \frac{s_4}{\theta} \right)\right] \right)\\
		&= e^{2 \pi i \frac{1}{\theta} s_1 s_4 } e^{ - 2 \pi i \{ - x^{-1} s_1 s_4\}_p }
		\end{align*}

		Note that for a $p$-adic number $x = \sum_{j = v}^\infty x_j p^j$ of order $v$,
		\[
		-x = (p - x_v)p^v + (p - 1 - x_{v+1}) p^{v+1} + \cdots + (p - 1 - x_j) p^j + \cdots
		\]
		It is clear that if $x \in \Z_p$, then $\{ x \}_p = \{ -x \}_p = 0$. If $x \in \Q_p \setminus \Z_p$, then $\{ x \}_p + \{ - x\}_p = 1$. It follows that for $-x^{-1} s_1 s_4 \in \Z_p$, either $\{ - x^{-1} s_1 s_4\}_p = \{ x^{-1} s_1 s_4\}_p$ or $\{ - x^{-1} s_1 s_4\}_p = 1 - \{ x^{-1} s_1 s_4\}_p$. In either case, we have 
		\begin{align*}
		\bar{\eta} \left( \left[(s_1, -s_1), \left( -x^{-1}s_2, \frac{s_2}{\theta} \right)\right], \left[(s_3, -s_3), \left( -x^{-1}s_4, \frac{s_4}{\theta} \right)\right] \right) &= e^{2 \pi i \frac{1}{\theta} s_1 s_4 } e^{ - 2 \pi i \{ - x^{-1} s_1 s_4\}_p }\\
		&= e^{2 \pi i \frac{1}{\theta} s_1 s_4 } e^{2 \pi i \{ x^{-1} s_1 s_4\}_p}
		\end{align*}

		By Theorem \ref{Xi_p_Psi_alpha}, the multiplier $\Psi_{\beta}$ on $\Z\left[ \frac{1}{p} \right] \times \Z\left[ \frac{1}{p} \right]$ is given by:
		\begin{align*}
		\Psi_\beta \left( \left( \frac{j_1}{p^{k_1}}, \frac{j_2}{p^{k_2}} \right), \left( \frac{j_3}{p^{k_3}} ,\frac{j_4}{p^{k_4}} \right) \right) &= \exp \left( 2 \pi i \beta_{(k_1+k_4)} j_1 j_4 \right)\\
		&= \exp \left( 2 \pi i \left( \dfrac{1}{\theta p^{k_1 + k_4}} + \dfrac{\sum_{j = 0}^{k_1 + k_4 - 1} y_j p^j}{p^{k_1 + k_4}} \right) j_1 j_4 \right)\\
		&= \exp \left( 2\pi i \dfrac{j_1 j_4}{\theta p^{k_1 + k_4}} \right) \cdot \exp \left( 2 \pi i \dfrac{\sum_{j = 0}^{k_1 + k_4 - 1} y_j p^j}{p^{k_1 + k_4}}  j_1 j_4 \right)\\
		&= \exp \left( 2\pi i \frac{1}{\theta} s_1 s_4 \right) \cdot \exp \left( 2 \pi i \left( \sum_{j = 0}^{k_1 + k_4 - 1} y_j p^j \right) s_1 s_4 \right)
		\end{align*} 
		
		By the Lemma \ref{fractional-part-congruence}, $\left\{ x^{-1} s_1 s_4 \right\}_p \equiv \left( \sum_{j = 0}^{k_1 + k_4 - 1} y_j p^j \right) \cdot s_1 s_4 \text{ modulo } \Z$, so we have
		\begin{align*}
		\exp \left( 2 \pi i \left( \sum_{j = 0}^{k_1 + k_4 - 1} y_j p^j \right) s_1 s_4 \right) = \exp \left( 2\pi i \left\{ x^{-1} s_1 s_4 \right\}_p  \right).
		\end{align*}
		
		This shows that $\bar{\eta} \left( \lambda_{x, \theta}(s_1,s_2), \lambda_{x, \theta}(s_3, s_4) \right) = \Psi_{\beta}\left( \left(s_1, s_2\right), \left( s_3, s_4 \right) \right)$, for all $s_1$, $s_2$, $s_3$, and $s_4$ in $\Z\left[ 1/p\right]$, as desired. Therefore, $C^\ast \left( D_{x, \theta}^\perp, \bar{\eta} \right)$ is isomorphic to $\A_{\beta}^\S$. It follows that $\A_\alpha^\S$ is Morita equivalent to $\A_{\beta}^\S$.
	\end{proof}
	
	Thus, one sees that given any $\alpha$ in $\Omega_p$ satisfying $\alpha_0 \neq 0$ and $x_\alpha$ invertible in the $p$-adic integers, we can find a $\beta \in \Omega_p$ (satisfying the same conditions), such that $\A_\alpha^\S$ and $\A_\beta^\S$ are Morita equivalent.
	It is then not unexpected that if we started with $\beta$, the formula in this theorem recovers $\alpha$ exactly.
	
	We conclude this section with the following slightly more general result, in which we only assume the $p$-adic integer $x_\alpha$ associated to $\alpha$ is nonzero.
	\begin{theorem}\label{beta-formula-general}
		Let $p$ be prime and let $\alpha = (\alpha_j)_{j \in \N} \in \Omega_p$ such that $p\alpha_{j+1} = \alpha_j + x_j$ with $\alpha_0 \neq 0$ and $x_\alpha = \sum_{j=0}^{\infty} x_j p^j \neq 0$.
		Then $C^\ast(D_{x, \theta}, \eta)$ is isomorphic to the noncommutative solenoid $\A_\alpha^\S$, where $x = x_\alpha$ and $\theta = \alpha_0$.
		Since $x_\alpha \neq 0$, we write (uniquely) $x_\alpha^{-1} = \sum_{j = -v}^{\infty} y_j p^j \in \Q_p$.
		Via the Heisenberg equivalence bimodule of Rieffel, $\A_\alpha^\S$ is Morita equivalent to a noncommutative solenoid $\A_\beta^\S$, where $\beta = \left( \beta_n \right)_{ n \in \N}$ is given by 
		\[
		\beta_n = \dfrac{1}{\theta p^n} + \dfrac{\sum_{j = -v}^{n-1} y_j p^j}{p^n}.
		\]
	\end{theorem}
	
	\begin{proof}
		Using Lemma \ref{fractional-part-congruence} and Lemma \ref{inverse-multiply-one} to their full generality, the proof follows the same procedure presented in the proof of Theorem \ref{beta-formula-nonzero}.
	\end{proof}
	

	\section{Directed systems of equivalence bimodules}\label{S-directed systems}
	
	\subsection{Directed system of equivalence bimodules.} In \cite{LP3}, a notion of \emph{directed system of equivalence bimodules} was introduced. We recall it here.
	
	\begin{definition}\label{directed-system}[\cite{LP3}, Definition 3.1]
		Let 
		\begin{center}
			\begin{tikzcd}
				A_0 \arrow[r, "\varphi_0"] & A_1 \arrow[r, "\varphi_1"] & A_2 \arrow[r, "\varphi_2"] & \cdots
			\end{tikzcd}
		\end{center}
	    and 
	    \begin{center}
	    	\begin{tikzcd}
	    		B_0 \arrow[r, "\psi_0"] & B_1 \arrow[r, "\psi_1"] & B_2 \arrow[r, "\psi_2"] & \cdots
	    	\end{tikzcd}
	    \end{center}
        be two directed systems of unital C*-algebras, whose *-morphisms are all unital maps. A sequence $(X_n, i_n)_{n \in \N}$ is a \emph{directed system of equivalence bimodule adapted to the sequence $(A_n)_{n \in \N}$ and $(B_n)_{n \in \N}$} when $X_n$ is an $A_n$-$B_n$-equivalence bimodule, whose $A_n$- and $B_n$-valued inner products are denoted respectively by $\inner{\cdot}{\cdot}_{A_n}$ and $\inner{\cdot}{\cdot}_{B_n}$, for all $n \in \N$, and such that the sequence
        \begin{center}
        	\begin{tikzcd}
        	X_0 \arrow[r, "i_0"] & X_1 \arrow[r, "i_1"] & X_2 \arrow[r, "i_2"] & \cdots
        	\end{tikzcd}
        \end{center}
        is a directed sequence of modules satisfying
        \begin{equation}\label{preserve-right-inner-product}
        	\inner{i_n(f)}{i_n(g)}_{B_{n+1}} = \psi_n\left( \inner{f}{g}_{B_n} \right), \quad \text{for all } f, g \in X_n,
        \end{equation}
        and
        \begin{equation}\label{preserve-right-action}
        	i_n(f \cdot b) = i_n(f) \cdot \psi_n(b), \quad \text{for all } f \in X_n, b \in B_n,
        \end{equation}
        with analogous but symmetric equalities holding for the $X_n$ viewed as left Hilbert $A_n$-modules:
        \begin{equation}\label{preserve-left-inner-product}
        	\inner{i_n(f)}{i_n(g)}_{A_{n+1}} = \varphi_n\left( \inner{f}{g}_{A_n} \right), \quad \text{for all } f, g \in X_n,
        \end{equation}
        and
        \begin{equation}\label{preserve-left-action}
        	i_n(a \cdot f) = \varphi_n(a) \cdot i_n(f) \quad \text{for all } f \in X_n, a \in A_n,
        \end{equation}
	\end{definition}
	
	As stated in \cite[Theorem 3.4]{LP3}, Definition \ref{directed-system} provides the necessary structure to construct an equivalence bimodule that implements the Morita equivalence of the direct limit C*-algebras.
	
	In this section, we present a construction for producing families of irrational noncommutative solenoids that are Morita equivalent to a given one.
	
	In short, first fix an irrational noncommutative solenoid $\A_\alpha^\S = \varinjlim A_{\alpha_{2n}}$ and a projection $P$ from an $m \times m$ matrix algebra over $A_{\alpha_0}$, both satisfying certain conditions.
	We can then form a noncommutative solenoid $\A_\beta^\S$ from the directed sequence of irrational rotation algebras $A_{\beta_{2n}} \cong P M_m\left( A_{\alpha_{2n}} \right)P$ with connecting maps given by Theorem \ref{direct-limit-realization}.
	We will show that $\varinjlim PA_{\alpha_{2n}}^m$, with the appropriate connecting maps, is a directed system of equivalence bimodules adapted to the sequences $\left( A_{\alpha_{2n}} \right)_{n \in \N}$ and $\left( P M_m\left(A_{\alpha_{2n}} \right)P \right)_{n \in \N}$.
	This allows us to state this section's main result as Theorem \ref{main-result}, which generalizes the results in Section 4 of \cite{LP3}.
	
	For the connecting maps between the approximating bimodules, it requires algebraically intensive steps to check that they satisfy Equations \ref{preserve-right-inner-product} through \ref{preserve-left-action}. Much of the current section is devoted to this.
	
	We begin by establishing the equivalence bimodules at each stage.
	
	\begin{notation}\label{nc-solenoid-notation}
		Let $\alpha = \left( \alpha_n \right)_{n \in \N} \in \Xi_p$ be an irrational sequence and $p \alpha_{n+1} = \alpha_n + x_n$ for all $n \in \N$.
		Let $\A_\alpha^\S = \varinjlim A_{\alpha_{2n}}$ be an irrational noncommutative solenoid realized as the direct limit of the sequence of irrational rotation algebras $\left( A_{\alpha_{2n}} \right)_{n \in \N}$.
		We denote the unique tracial state on $A_{\alpha_{2n}}$ by $\tau_{\alpha_{2n}}$.
		We use the same notation $\tau_{\alpha_{2n}}$ for the standard \emph{unnormalized} trace on $M_m\left( A_{\alpha_{2n}} \right)$, $m \geq 1$.
		Note that $\left( 1/m \right) \tau_{\alpha_{2n}}$ is then the unique tracial state on $M_m\left( A_{\alpha_{2n}} \right)$.
	\end{notation}
	
	Let $P$ be a nontrivial projection in $M_m\left( A_{\alpha_0} \right)$ with $\tau_{\alpha_0} \left( P \right) = c_0 \alpha_0 + d_0 \in (0, m)$ that satisfies the following condition:
	\begin{condition}\label{the-condition}
		$\gcd\left( c_0 p, d_0 - c_0x_0 \right) = 1$.
	\end{condition}
	We follow the convention that $\gcd\left( a, 0 \right) = a$ for any positive integer $a$.
	For details of the existence and construction of such a projection, see \cite{Rieffel81}.
	When $P$ is considered as a projection in $M_m\left( A_{\alpha_{2n}} \right)$ for any $n \in \N$, it is not hard to see that $\tau_{\alpha_{2n}} \left( P \right) =  \tau_{\alpha_0} \left( P \right) = c_0 \alpha_0 + d_0$, but $\alpha_0 = p^{2n}\alpha_{2n} - \sum_{j=0}^{2n-1} x_j p^j$, so the trace of $P$ there is 
	\[
	\tau_{\alpha_{2n}} \left( P \right) = \left( c_0 p^{2n} \right) \alpha_{2n} + \left( d_0 - c_0 \sum_{j = 0}^{2n - 1} x_j p^j \right).
	\]
	We let 
	\[
	c_{2n} = c_0 p^{2n} \quad \text{and} \quad d_{2n} = d_0 - c_0 \sum_{j = 0}^{2n - 1} x_j p^j,
	\]
	so that $\tau_{\alpha_{2n}} \left( P \right) = c_{2n} \alpha_{2n} + d_{2n}$ in $M_m\left( A_{\alpha_{2n}} \right)$.
	
	The follow lemma states Condition \ref{the-condition} in a form that will be useful to us later. 
	
	\begin{lemma}\label{condition-useful-form}
		If $\gcd\left( c_0 p, d_0 - c_0 x_0 \right) = 1$, then $\gcd\left( c_{2n}, d_{2n} \right) = 1$ for all $n \in \N$.
	\end{lemma}
	
	\begin{proof}
		First, it is clear that $\gcd\left( c_0 p, d_0 + c_0 x_0 \right) = 1$ implies $\gcd(c_0, d_0) = 1$, as any common divisor of $c_0$ and $d_0$ divides both $c_0p$ and $d_0 - c_0x_0$.
		
		Fix $n \geq 1$. We show that if $\gcd\left( c_0p^{2n}, d_0 - c_0 \sum_{j=0}^{2n-1} x_j p^j\right) > 1$, then $\gcd\left( c_0p, d_0 - c_0x_0 \right) > 1$.
		It then follows that $\gcd(c_0p, d_0-c_0x_0) = 1$ implies $\gcd(c_0p^{2n}, d_0 - c_0 \sum_{j=0}^{2n-1} x_j p^j  ) = 1$ for all $n \geq 1$.
		
		Assume $\gcd\left( c_{2n}, d_{2n} \right) = \gcd\left( c_0p^{2n}, d_0 - c_0 \sum_{j=0}^{2n-1} x_j p^j \right) > 1$, then either $p$ divides $d_0 - c_0 \sum_{j=0}^{2n-1} x_j p^j$, or $q$ divides $d_0 - c_0 \sum_{j=0}^{2n-1} x_j p^j$ for some divisor $q$ of $c_0$ that is strictly greater than $1$.
		In the former case, for some integer $k$, 
		\begin{align*}
			kp = d_0 - c_0 \sum_{j=0}^{2n-1} x_j p^j = \left(d_0 - c_0 x_0 \right) - c_0 \sum_{j=1}^{2n-1} x_j p^j,
		\end{align*}
		so $d_0 - c_0 b_0 = kp + c_0\sum_{j=1}^{2n-1} x_j p^j = p\left( k + c_0\sum_{j = 1}^{2n-1} x_j p^{j-1} \right)$.
		This makes $p$ a common divisor of $c_0p$ and $d_0-c_0x_0$. In the latter case, $d_0 - c_0x_0 = \left( d_0 - c_0 \sum_{j=0}^{2n-1} x_j p^j \right) + c_0\sum_{j=1}^{2n-1} x_j p^j$, making $q$ a common divisor of $d_0-c_0 x_0$ and $c_0p$.
		This shows that $\gcd(c_0p, d_0-c_0x_0) = 1$ implies $\gcd(c_0p^{2n}, d_0 - c_0 \sum_{j=0}^{2n-1} x_j p^j  ) = 1$ for all $n \geq 1$.
	\end{proof}
	
	\begin{remark}\label{condition-special-case}
		We single out a special occurrence of how Condition \ref{the-condition} can be satisfied: if $d_0 = 0$, then Condition \ref{the-condition} is satisfied if and only if both $c_0 = 1$ and $x_0 \neq 0$. This will be exactly the situation described in Proposition \ref{relate-constructions}.
	\end{remark}

	\subsection{Review of Rieffel's standard bimodule between irrational rotation algebras}
	
	\begin{notation}\label{notation-a-b-c-d}
		Throughout this section, we fix an irrational number $\alpha$ and let $c$ and $d$ be a pair of integers that generate $\Z$ with $c\alpha + d \neq 0$ and $c \neq 0$.
		Set $\gamma = 1/\left( c\alpha + d \right)$.
		Moreover, we let $a$ and $b$ be any two integers such that $ad - bc = 1$, and let $\beta = \left( a \alpha + b \right)\gamma = \left( a\alpha + b \right)/\left( c\alpha + d \right)$.
	\end{notation}
	
	For any $k \geq 1$, let $V\left( d, c; k \right)$ be the \textit{standard $M_k\left( A_\beta \right)$-$A_\alpha$-equivalence bimodule} defined by Rieffel in \cite{Rieffel83}. We start with a lemma that provides an alternative presentation of this bimodule.
	
	\begin{lemma}\cite[Lemma 6]{Kodaka92}\label{iso-right-module}
		With notation as above, let $P$ be a projection in $M_m\left( A_\alpha \right)$, $m \geq 1$, with unnormalized trace $k\left( c\alpha + d \right)$, then $V\left( d, c; k \right)$ is isomorphic to $PA_\alpha^m$ as a right $A_\alpha$-module.
	\end{lemma}
	
	Since $P$ is a nontrivial projection in the simple C*-algebra $M_m\left( A_\alpha \right)$, $P$ is full and $PA_\alpha^m$ is a $P M_m\left( A_\alpha \right)P$-$A_\alpha$-equivalence bimodule.
	It follows from this observation and the lemma above that $P M_m\left( A_\alpha \right) P \cong M_k\left( A_\beta \right)$ as C*-algebras.
	
	We now recall Rieffel's construction of the standard $M_k\left( A_\beta \right)$-$A_\alpha$-equivalence bimodule for the case $k = 1$, which is sufficient for our purposes.
	The rest of this section is based on \cite[Theorem 1.1]{Rieffel83} and its proof.
	It is an application of Phil Green's Symmetric Imprimitivity Theorem that is proved in \cite[Situation 10]{Rieffel82}, or see \cite{Rieffel76}.
	
	\begin{theorem}\cite[Theorem 1.1]{Rieffel83}\label{alpha-beta-bimodule-structure}
		With notation as above, let $G = \R \times \Z_c$ (If $c$ is a negative integer, we still define $\Z_c$ to be $\Z / c\Z$) and consider the following subgroups of $G$:
		\[
		 H  = \{ \left( n, [dn]_c \right) : n \in \Z \}, \qquad K = \left\{ \left( n \gamma, [n]_c \right) : n \in \Z \right\}.
		\]
		Let $H$ act on $K \backslash G$ (the right cosets of $K$) by right translation, and let $K$ act on $G / H$ (the left cosets of $H$) by left translation.
		Then the transformation group C*-algebras $C^\ast\left( H, K \backslash G \right)$ and $C^\ast\left( K,  G / H \right)$ are isomorphic to $A_\alpha$ and $A_\beta$, respectively. Furthermore, $C_c\left( G \right)$, suitably completed and structured, provides an $A_\beta$-$A_\alpha$-equivalence bimodule.
	\end{theorem}
	
	Both isomorphisms rely on identifying the quotient group with the circle group $\T$ and realizing the corresponding group action of $\Z$ (both $H$ and $K$ are isomorphic to $\Z$) gives rise to a rotation algebra.
	In particular, we employ the following identification of $K \backslash G$ and $G / H$ with $\T$.
	The map given by
	\[
	\begin{cases}
		G &\rightarrow \T\\
		(t, [m]_c) &\mapsto \dfrac{t/\gamma - m}{c},
	\end{cases}
	\]
	is a group homomorphism with kernel exactly $K$. We can then identify $K \setminus G$ with $\T$.
	Similarly, the map given by
	\[
	\begin{cases}
		G &\rightarrow \T\\
		(t, [m]_c) &\mapsto \dfrac{t - am}{c},
	\end{cases}
	\]
	is a group homomorphism with kernel exactly $H$, so we can identify $G/H$ with $\T$.\\
	
	\begin{remark}
		We use an inverse identification of $G/H$ with $\T$ than what is presented in \cite{Rieffel83}.
		Over there, the identification is derived from the map $G \rightarrow \T$ given by $(t, [m]_c) \mapsto (am - t)/c$.
		One of reasons for choosing this inverse identification is to later have $U_\beta V_\beta = e^{2\pi i \beta} V_\beta U_\beta$, instead of $e^{2\pi i \beta} U_\beta V_\beta = V_\beta U_\beta$, where $U_\beta$ and $V_\beta$ are the usual generating unitaries of $A_\beta$.
	\end{remark}
	
	In the following two lemmas, we record the formulas for the generators of $A_\beta$, the left $A_\beta$-action on $C_c\left( \R \times \Z_c \right)$, and the $A_\beta$-valued inner product, all under the present identification of $G / H \times K$ with $\T \times \Z$.
	To ease notation, we just write $n$ to represent $\left( n\gamma, [n]_c \right) \in K$.
	These formulas are derived originally in \cite{Rieffel76} and applied to (matrix algebras over) rotation algebras in \cite{Rieffel83}.
	In the proof of \cite[Proposition 4.2]{LP3}, they are given for the case where $a = 1, b = 0, c = p^{2j}$, and $d = 1$, for any $j \in \N$. 
	
	For $(t, [m]_c) \in \R \times \Z_c$, we denote by $\tilde{(t, [m]_c)}$ the class of $(t, [m]_c)$ in $G/H$.
	In the second lemma, the left $A_\beta$-action is define on the dense subspace $C_c\left( G / H \times K \right)$ of $C^\ast \left( K, G/H \right) \cong A_\beta$, and $(r, k) \in \R \times \Z$ is identified with $G/H \times K$ via the map $(e^{2\pi i r}, k) \mapsto \left( \tilde{(cr, [0]_c)}, (k\gamma, [k]_c) \right) \in G/H \times K$.
	
	\begin{lemma}
		The generators $U_\beta$ and $V_\beta$ of $A_\beta$ are given by 
		\[
		U_\beta( \tilde{\left( t, [m]_c \right)} , n) = 
		\begin{cases}
			0, &\text{ if } n \neq 1\\
			1, &\text{ if } n  = 1
		\end{cases}
		\]
		and 
		\[
		V_\beta( \tilde{\left( t, [m]_c \right)} , n) = 
		\begin{cases}
			0, &\text{ if } n \neq 0\\
			e^{2 \pi i \left[ (t - am) / c\right]}, &\text{ if } n  = 0.
		\end{cases}
		\]
	\end{lemma}
	We compute that 
	$$
	U_\beta V_\beta \left(  \tilde{\left( t, [m]_c \right)} , n \right) = 
	\begin{cases}
		0, &\text{ if } n \neq 1,\\
		e^{2 \pi i (t-\gamma - am + a) / c}, &\text{ if } n  = 1. 
	\end{cases}
	$$
	and
	$$
	V_\beta U_\beta \left(  \tilde{\left( t, [m]_c \right)} , n \right) = 
	\begin{cases}
		0, &\text{ if } n \neq 1,\\
		e^{2 \pi i (t - am) / c}, &\text{ if } n  = 1. 
	\end{cases}
	$$
	It follows that $U_\beta V_\beta = e^{2 \pi i (a - \gamma)/c} V_\beta U_\beta$. After writing $1$ as $ad - bc$, a few steps of algebra will show that $(a - \gamma)/c = \beta$, so we have $U_\beta V_\beta = e^{2\pi i \beta} V_\beta U_\beta$.
	
	\begin{lemma}
		For any $F_1, F_2 \in C_c\left( \R \times \Z_c \right)$, $f \in C_c\left( G/H \times K \right)$, $(r, k) \in \left( \R, \Z \right)$ and $(t, [m]_c) \in \R \times \Z_c$,
		\begin{equation}\label{left-inner-product}
			\inner{F_1}{F_2}_{A_\beta} (r, k) = \sum_{m \in \Z} F_1\left( cr + m, [dm]_c \right) \bar{F_2\left( cr + m - k\gamma, [dm-k]_c \right)},
		\end{equation}
		and
		\begin{equation}\label{left-action}
			\left(f \cdot F \right) \left( t, [m]_c \right) = \sum_{n \in \Z} f \left( \tilde{(t, [m]_c)}, (n\gamma, [n]_c) \right) F\left( t - n\gamma, [m-n]_c \right).
		\end{equation}
		In particular, it follows that the generators $U_\beta$ and $V_\beta$ act on $C_c(G)$ in the following way:
		\begin{equation}\label{left-Ubeta-action}
			\left( U_\beta \cdot F \right)(t, [m]_c) = F \left( t - \gamma, [m-1]_c \right),
		\end{equation}
		and
		\begin{equation}\label{left-Vbeta-action}
			\left( V_\beta \cdot F \right)(t, [m]_c) = e^{2 \pi i \left( t - am \right)/c} F(t, [m]_c).
		\end{equation}
	\end{lemma}
	
	The right $A_\alpha$-module structure is given by the following lemmas.
	
	\begin{lemma}
		The generators $U_\alpha$ and $V_\alpha$ of $A_\alpha$ are given by 
		\[
		U_\alpha( \tilde{\left( t, [m]_c \right)} , n) = 
		\begin{cases}
			0, &\text{ if } n \neq 1\\
			1, &\text{ if } n  = 1
		\end{cases}
		\]
		and 
		\[
		V_\alpha( \tilde{\left( t, [m]_c \right)} , n) = 
		\begin{cases}
			0, &\text{ if } n \neq 0\\
			e^{2 \pi i \left[ (t/\gamma - m) / c\right]}, &\text{ if } n  = 0.
		\end{cases}
		\]
	\end{lemma}
	
	\begin{lemma}
		For any $F_1, F_2 \in C_c\left( \R \times \Z_c \right)$, $g \in C_c\left( H \times K \backslash G \right)$, $(r, k) \in \left( \R, \Z \right)$ and $(t, [m]_c) \in \R \times \Z_c$,
		\begin{equation}\label{right-inner-product}
			\inner{F_1}{F_2}_{A_\alpha} (r, k) = \sum_{m \in \Z} \bar{ f \left( \left( c r - m \right)\gamma, [-m]_c \right) } g \left( \left( cr - m \right)\gamma + k, [dk - m]_c \right),
		\end{equation}
		and
		\begin{equation}\label{right-action}
			\left(F \cdot g \right) \left( t, [m]_c \right) = \sum_{n \in \Z} F \left( t-n, [m-dn]_c \right) g\left( \widetilde{(t - n, [m - dn]_c)} , \left( n, [dn]_c \right) \right).
		\end{equation}
		In particular, it follows that the generators $U_\alpha$ and $V_\alpha$ act on $C_c(G)$ in the following way:
		\begin{equation}\label{right-Ualpha-action}
			\left( F \cdot U_\alpha\right) (t, [m]_c) = F \left( t - 1, [m-d]_c \right),
		\end{equation}
		and
		\begin{equation}\label{right-Valpha-action}
			\left( F \cdot V_\alpha \right)(t, [m]_c) = e^{2 \pi i \left( t/\gamma - m \right)/c} F(t, [m]_c).
		\end{equation}
	\end{lemma}

	\subsection{Explicit construction of equivalence bimodules at each stage.}
	
	\begin{proposition}\label{construction-beta}
		We follow the notation established in Notation \ref{nc-solenoid-notation} .
		Let $P$ be a projection in $M_m\left( A_{\alpha_0} \right)$ with $\tau_{\alpha_{2n}}\left( P \right) = c_0 \alpha_0 + d_0 = c_{2n} \alpha_{2n} + d_{2n}$.
		Assume $P$ satisfies Condition \ref{the-condition}.
		For each $n \in \N$, let $a_{2n}$ and $b_{2n}$ be a pair of integers satisfying $a_{2n} d_{2n} - b_{2n} c_{2n} = 1$ and set 
		\[
		\beta_{2n} = \dfrac{a_{2n} \alpha_{2n} + b_{2n} }{c_{2n} \alpha_{2n} + d_{2n}}.
		\]
		In particular, we can uniquely choose $a_{2n}$ and $b_{2n}$ so that $\beta_{2n}$ is between $0$ and $1$.
		Then for each $n \in \N$, $C_c\left( \R \times \Z_{c_{2n}} \right)$, suitably completed, has the structure of a $A_{\beta_{2n}}$-$A_{\alpha_{2n}}$-equivalence bimodule.
		
		Moreover, we can form a noncommutative solenoid $\A_\beta^\S$ by taking the direct limit of irrational rotation algebras $A_{\beta_{2n}}$ with the appropriate connecting maps.
	\end{proposition}
	
	\begin{proof}
		Since each connecting map $\varphi_n : A_{\alpha_{2n}} \rightarrow A_{\alpha_{2n + 2}}$ is a unital inclusion, we can view $P$ as a projection in $M_m\left( A_{\alpha_{2n}} \right)$ for all $n$.
		By Lemma \ref{condition-useful-form}, $c_{2n}$ and $d_{2n}$ together generate $\Z$.
		The $A_{\beta_{2n}}$-$A_{\alpha_{2n}}$-equivalence bimodule structure is then the one from Theorem \ref{alpha-beta-bimodule-structure}.
		
		Since $c_{2n}$ and $d_{2n}$ are coprime, we can find integers $a_{2n}$ and $b_{2n}$ such that $a_{2n} d_{2n} - b_{2n} c_{2n} = 1$. 
		Moreover, if $a_{2n}$ and $b_{2n}$ are a pair of such integers, as are $a_{2n} + \ell c_{2n}$ and $b_{2n} + \ell d_{2n}$ for any integer $\ell$.
		If $\beta_{2n} = \left( a_{2n} \alpha_{2n} + b_{2n} \right) / \left( c_{2n} \alpha_{2n} + d_{2n} \right)$, then
		\[
		\dfrac{\left( \alpha_{2n} + \ell c_{2n} \right) + \left( b_{2n} + \ell d_{2n} \right)}{ c_{2n} \alpha_{2n} + d_{2n} } = \beta_{2n} + \ell. 
		\] 
		It follows that there is a unique pair of $a_{2n}$ and $b_{2n}$ putting $\beta_{2n}$ is between $0$ and $1$.
		
		We show that the family of irrational rotation algebras $\left( A_{\beta_{2n}} \right)_{n \in \N}$ form a noncommutative solenoid by taking the direct limit with the prescribed connecting maps.
		To this end, we simply need to check that $p^2 \beta_{2n + 2} \equiv \beta_{2n} \mod \Z$ for all $n \in \N$.
		
		To ease notation, we denote $a_{2n}$, $b_{2n}$, $c_{2n}$, $d_{2n}$, $a_{2n+2}$, $b_{2n+2}$, $c_{2n+2}$, and $d_{2n+2}$ by $a$, $b$, $c$, $d$, $A$, $B$, $C$, and $D$, respectively.
		By definition of $\alpha$, $C = cp^2$ and $D = d - ck$ for some of $k \in \{ 0, \dotsc, p^2 - 1 \}$. We have
		\[
		\beta_{2n} = \dfrac{a\alpha_{2n}+ b}{c\alpha_{2n} + d}, \quad \text{and} \quad \beta_{2n+2} = \dfrac{A \alpha_{2n+2} + B}{C\alpha_{2n+2} + D} = \dfrac{A \alpha_{2n+2} + B}{(cp^2)\alpha_{2n+2} + (d - kc)}.
		\]
		We show that $p^2\beta_{2n+2} - \beta_{2n}$ is an integer.
		With a few steps of algebra, one checks that
		\[
		p^2 \beta_{2n+2} - \beta_{2n} = \dfrac{(A - a)\alpha_{2n} + (Ak + Bp^2 - b)}{c\alpha_{2n} + d}.
		\]
		Note that $Ad - Ack - Bcp^2 = 1$, so $Ak = \left( Ad - Bcp^2 - 1 \right) / c$, and arrive at the following simplification:
		\[
		\dfrac{(A - a)\alpha_{2n} + (Ak + Bp^2 - b)}{c\alpha_{2n} + d} = \dfrac{A-a}{c}.
		\]
		Since $1 = ad - bc = Ad - \left(Ak + Bp^2\right)c$, it must be the case that $A-a = \ell c$ for some integer $\ell$. Thus, $p^2 \beta_{2n+2} \equiv \beta_{2n} \mod \Z$.
		Note that if both $\beta_{2n}$ and $\beta_{2n+2}$ are chosen to be between $0$ and $1$, then $p^2 \beta_{2n+2} - \beta_{2n} \in \{ 0, \dotsc, p^2 - 1 \}$.
	\end{proof}
	
	\begin{remark}
		In the above proposition, the sequence $\left( \beta_{2n} \right)_{n \in \N}$ determines a unique element $\beta = \left( \beta_k \right)_{k \in \N}$ of $\Xi_p$ by Remark \ref{determine-alpha}. By the same remark, all other choices for the $a_{2n}$'s and $b_{2n}$'s give rise to sequences in $\Omega_p$ that lead to the same noncommutative solenoid $\A_\beta^\S$.
	\end{remark}
	
	Let $X_{2n} = \bar{C_c\left( \R \times \Z_{c_{2n}} \right)}$ denote the $A_{\beta_{2n}}$-$A_{\alpha_{2n}}$-equivalence bimodule given in the above proposition.
	When working with $X_{2n}$, results are almost always formulated in terms for the following dense subset.
	For $f \in C_c\left(\R\right)$ and $j \in \{ 0, \dotsc, c_{2n} - 1 \}$, let $f\delta_j$ be the function in $C_c\left( \R \times \Z_{c_{2n}} \right)$ given by
	\[
	f \delta_j \left( t, [m] \right) = 
	\begin{cases}
		f \left( t \right) & \text{ if } [m] = [j] \text{ in } \Z_{c_{2n}},\\
		0 &\text{ otherwise}. 
	\end{cases}
	\]
	One can regard $\delta_j$ as the indicator function on $[j] \in \Z_{c_{2n}}$.
	Then the set
	\[
	\left\{ f \delta_j \:: j \in f \in C_c\left(\R\right) \text{ and } j \in \{ 0, \dotsc, c_{2n} - 1 \}  \right\}
	\]
	spans a dense subalgebra of $C_c\left( \R \times \Z_{c_{2n}} \right)$.

	We now define the key embedding of $X_{2n}$ in $X_{2n+2}$ for each $n$ that gives us the directed systems of equivalence bimodules.
	
	\begin{definition}\label{define-bimodule-embedding}
		We define $\iota_n : X_{2n} \rightarrow X_{2n+2}$ on functions of the form $f \delta_j$ and extend by linearity (Since $\Z_{c_{2n}}$ is finite, extending by linearity is especially simple):
		\begin{equation}
			\iota_n(f \delta_j)(t, [m]) = \dfrac{1}{\sqrt{p}} \sum_{i = 0}^{p-1} f\left( \dfrac{t}{p} \right) \delta_{jp + i c_0 p^{2n+1}}([m]),
		\end{equation}
		for $t \in \R$ and $[m] \in \Z_{c_{2n + 2}}$. That is, $\iota_n(f \delta_j)(t, [m])$ evaluates to $\left( 1/\sqrt{p} \right)f(t/p)$ when $m = jp + i c_0 p^{2n+1} + \ell p^{2n+2}$ for $i \in \{ 0, \dotsc, p-1 \}$ and $\ell \in \Z$, and zero otherwise. 
	\end{definition}
	
	It remains to show that $\left( X_{2n}, \iota_n \right)_{n \in \N}$ satisfies Definition \ref{directed-system}.
	Specifically, we need to check that Equations \ref{preserve-right-inner-product} - \ref{preserve-left-action} are satisfied.
	
	\begin{notation}
		In the remainder of this section, for $\A_\alpha^\S = \varinjlim \left( A_{\alpha_{2n}}, \psi_n \right)_{n \in \N}$, we let $\A_\beta^\S = \varinjlim \left( A_{\beta_{2n}}, \varphi_n \right)_{n \in \N}$ be the irrational noncommutative solenoid constructed in Proposition \ref{construction-beta}.
		Both embeddings $\varphi_n : A_{\beta_{2n}} \rightarrow A_{\beta_{2n + 2}}$ and $\psi_n : A_{\alpha_{2n}} \rightarrow A_{\alpha_{2n + 2}}$ are as given in Lemma \ref{rotation-algebra-embedding-formula}.
		Lastly, for any $f \in C_c\left( \R \right)$, we let $f_{\xi}(t)$ be $f(t + \xi)$ for any fixed $\xi \in \R$. 
	\end{notation}

	\begin{lemma}\label{iota-preserves-left-action}
		For each $n \in \N$, $\iota_n: X_{2n} \rightarrow X_{2n+2}$ preserves left module action. That is,
		\begin{equation}
			\iota_{n} \left( f \cdot F \right) = \varphi_n\left( f \right) \cdot \iota_n \left( F \right),
		\end{equation}
		for all $f \in A_{\beta_{2n}}$ and $F \in X_{2n}$.
	\end{lemma}
	
	\begin{proof}
		We check that $\iota_{n}$ respects left module action by showing that for any $f \in C_c\left( \R \right)$ and $j \in \{ 0, \dotsc, c_{2n} - 1 \}$,
		\[
		\iota_{n} \left( U_{\beta_{2n}} \cdot \left( f \delta_j \right) \right) = U_{\beta_{2n+2}}^p \cdot \iota_{n}\left( f \delta_j \right),\text{ and }
		\iota_{n} \left( V_{\beta_{2n}} \cdot \left( f \delta_j \right) \right) = V_{\beta_{2n+2}}^p \cdot \iota_{n}\left( f \delta_j \right).
		\]
		
		Since $\tau_{\alpha_{2n}} \left( P \right) = c_0 \alpha_0 + d_0 = c_{2n} \alpha_{2n} + d_{2n}$, we have $\gamma_{2n} = 1/ \left( c_{2n} \alpha_{2n} + d_{2n} \right) = 1/ \left(  c_0 \alpha_0 + d_0 \right)$ (see Notation \ref{notation-a-b-c-d}), and we simply denote this number by $\gamma$.
		
		For any $\left( t, [m] \right) \in \R \times \Z_{c_{2n+2}}$, it follows from Equation \ref{left-Ubeta-action} that
		
		\begin{align*}
			\iota_{n}\left( U_{\beta_{2n}} \cdot \left( f \delta_j \right) \right)\left(t, [m] \right) &= \iota_{n}\left( f_{-\gamma} \delta_{j+1} \right) \left(t, [m] \right)\\
			&= \dfrac{1}{\sqrt{p}}\sum_{i = 0}^{p-1} f_{-\gamma} \delta_{(j+1)p+ i c_0 p^{2n+1}} \left(  \dfrac{t}{p}, [m] \right).
		\end{align*}
		
		On the other hand, 
		\begin{align*}
			U_{\beta_{2n+2}}^p \cdot \iota_{n} \left( f  \delta_j \right) \left( t, [m] \right) &= \iota_{n}\left( f \delta_j \right)\left( t - p\gamma, [m-p] \right)\\
			&= \dfrac{1}{\sqrt{p}}\sum_{i = 0}^{p-1} f \delta_{jp + i c_0 p^{2n+1}} \left( \frac{t-p\gamma}{p}, [m-p] \right)\\
			&= \dfrac{1}{\sqrt{p}} \sum_{i = 0}^{p-1} f_{-\gamma} \delta_{jp + i c_0 p^{2n+1} + p} \left( \frac{t}{p}, [m] \right).
		\end{align*}
		
		For any $(t, [m]) \in \R \times \Z_{c_{2n}}$, by Equation \ref{left-Vbeta-action},
		\[
		\left( V_{\beta_{2n}} \cdot \left( f \delta_j \right) \right) (t, [m]) = \begin{cases}
			e^{2 \pi i \left( t - a_{2n}j \right)/c_{2n}} f(t) &\text{if } [m] = [j] \text{ in } \Z_{c_{2n}},\\
			0 &\text{otherwise}.
		\end{cases}
		\]
		Then for $\left( t, [m] \right) \in \R \times \Z_{c_{2n+2}}$,
		\[
		\iota_{n}\left( V_{\beta_{2n}} \cdot \left( f \delta_j \right) \right)\left(t, [m] \right) = \left( 1/\sqrt{p} \right)e^{2 \pi i (t/p - a_{2n} j )/c_{2n}}f(t/p)
		\]
		when $m = jp + i c_0 p^{2n+1} + \ell p^{2n+2}$ for any $i \in \{ 0, \dotsc, p-1 \}$ and $\ell \in \Z$.
		Otherwise it evaluates to zero.
		
		On the other hand, using the fact that $p c_{2n} = c_0 p^{2n+1} = c_{2n+2}/p$, we have 
		\begin{align*}
			V_{\beta_{2n+2}}^p \cdot \iota_{n} \left( f \delta_j \right) \left( t, [m] \right) &= e^{2 \pi i \left[ (t - a_{2n+2}m)/c_{2n+2} \right]\cdot p} \iota_{n}\left( f \delta_j \right)\left( t, [m] \right)\\ 
			&= \dfrac{1}{\sqrt{p}} e^{2 \pi i (t - a_{2n+2}m)/p c_{2n}} \sum_{i=0}^{p-1} f \delta_{jp + i c_0 p^{2n+1}} \left( \frac{t}{p}, [m] \right),
		\end{align*}
		which is nonzero for $m = jp + i c_0 p^{2n+1} + \ell p^{2n+2}$ for any $i \in \{ 0, \dotsc, p-1 \}$ and $\ell \in \Z$.
		Notice that the evaluation is independent of the choice of representative for $[m]_{c_{2n+2}}$, so we let $\ell = 0$ and 
		\begin{align*}
			\left( 1/\sqrt{p} \right)e^{2 \pi i \left[ t - a_{2n+2}\left( jp + i c_0 p^{2n+1} \right)/pc_{2n} \right]} &= \left( 1/\sqrt{p} \right)e^{2 \pi i \left[ t - a_{2n+2}\left( jp + i c_0 p^{2n+1} \right)/ \left( c_0 p^{2n+1} \right) \right]}
		\end{align*}
		Moreover, since $a_{2n+2} = a_{2n} + z c_0 p^{2n}$ for some $z \in \Z$,
		\begin{align*}
			\dfrac{t - a_{2n+2}\left( jp + i c_0 p^{2n+1} \right)}{ c_0 p^{2n+1}} &= \dfrac{t - \left( a_{2n} + z c_0 p^{2n} \right)\left( jp + i c_0 p^{2n+1} \right)}{ c_0 p^{2n+1}}\\
			&= \dfrac{t - a_{2n}jp }{c_0 p^{2n+1}} + K_{2n}\\
			&= \dfrac{t/p - a_{2n}j}{c_0 p^{2n}} + K_{2n}\\
			&= \dfrac{t/p - a_{2n}j}{c_{2n}} + K_{2n}
		\end{align*}
		with integer $K_{2n} = -a_{2n} - zj - z i c_0 p^{2n}$.
		This shows that 
		$$
		e^{2 \pi i \left( t - a_{2n+2}\left( jp + i c_0 p^{2n+1} \right)/ p c_{2n} \right)} = e^{2 \pi i \left( t/p-a_{2n}j \right)/c_{2n}}
		$$ 
		and allows us to the conclude that $\iota_{n} \left( V_{\beta_{2n}} \cdot \left( f \delta_j \right) \right) = V_{\beta_{2n+2}}^p \cdot \iota_{n}\left( f \delta_j \right)$.
	\end{proof}

	\begin{lemma}\label{iota-preserves-right-action}
		For each $n \in \N$, $\iota_n: X_{2n} \rightarrow X_{2n+2}$ preserves right module action. That is,
		\begin{equation}
			\iota_{n} \left( F \cdot g \right) = \iota_n \left( F \right) \cdot \psi_n\left( g \right),
		\end{equation}
		for all $g \in A_{\alpha_{2n}}$ and $F \in X_{2n}$.
	\end{lemma}
	
	\begin{proof}
		We check that $\iota_{n}$ respects left module action by showing that for any $f \in C_c\left( \R \right)$ and $j \in \{ 0, \dotsc, c_{2n} - 1 \}$,
		\[
		\iota_{n} \left( \left( f \delta_j \right) \cdot U_{\alpha_{2n}} \right) = \iota_{n}\left( f \delta_j \right) \cdot U_{\alpha_{2n+2}}^p,\text{ and }
		\iota_{n} \left( \left( f \delta_j \right) \cdot V_{\alpha_{2n}} \right) = \iota_{n}\left( f \delta_j \right) \cdot V_{\alpha_{2n+2}}^p.
		\]
		
		Since $\tau_{\alpha_{2n}} \left( P \right) = c_0 \alpha_0 + d_0 = c_{2n} \alpha_{2n} + d_{2n}$, we have $\gamma_{2n} = 1/ \left( c_{2n} \alpha_{2n} + d_{2n} \right) = 1/ \left(  c_0 \alpha_0 + d_0 \right)$ (see Notation \ref{notation-a-b-c-d}), and we simply denote this number by $\gamma$.
		
		For any $\left( t, [m] \right) \in \R \times \Z_{c_{2n+2}}$, it follows from Equation \ref{right-Ualpha-action} that
		
		\begin{align*}
			\iota_n \left( \left( f \delta_j \right) \cdot U_{\alpha_{2n}} \right) \left( t, [m] \right) &= \iota_n \left( f_{-1} \delta_{j + d_{2n}} \right)  \left( t, [m] \right)\\
			&= \dfrac{1}{\sqrt{p}} \sum_{i = 0}^{p-1} f_{-1} \delta_{ \left( j + d_{2n} \right) p + i c_0 p^{2n+1} } \left( \dfrac{t}{p}, [m] \right)
		\end{align*}
		
		On the other hand, 
		
		\begin{align*}
			\left( \iota_n \left( f \delta_j \right) \cdot U_{\alpha_{2n + 2}}^p \right) \left( t, [m] \right) &= \iota_n\left(  f \delta_j \right) \left( t - p, [m - pd_{2n+2}] \right)\\
			&= \dfrac{1}{\sqrt{p}}\sum_{i = 0}^{p-1} f \delta_{jp + i c_0 p^{2n+1}} \left( \dfrac{t-p}{p}, [m - pd_{2n+2}] \right)\\
			&= \dfrac{1}{\sqrt{p}}\sum_{i = 0}^{p-1} f_{-1} \delta_{jp + i c_0 p^{2n+1} + pd_{2n+2}} \left( \dfrac{t}{p}, [m] \right)
		\end{align*}
		
		Since $d_{2n+2} = d_{2n} - c_0 x_{2n} p^{2n} - x_{2n+1}p^{2n+1}$,
		\[
		jp + i c_0 p^{2n+1} + pd_{2n+2} \equiv jp + i c_0 p^{2n+1} + pd_{2n} + c_0 x_{2n} p^{2n+1} \mod p^{2n+2}
		\]
		
		For $i = 0, \dotsc, p-1$, $i c_0 p^{2n+1}$ goes through the same equivalence classes of $\Z_{c_0 p^{2n+2}} = \Z_{c_{2n+2}}$ as $\left( i + x_{2n} \right) c_0 p^{2n+1}$, so 
		\[
		\delta_{ \left( j + d_{2n} \right) p + i c_0 p^{2n+1} } = \delta_{jp + i c_0 p^{2n+1} + pd_{2n+2}},
		\]
		and $\iota_{n} \left( \left( f \delta_j \right) \cdot U_{\alpha_{2n}} \right) = \iota_{n}\left( f \delta_j \right) \cdot U_{\alpha_{2n+2}}^p$.
		
		For any $\left( t, [m] \right) \in \R \times \Z_{c_{2n}}$, by Equation \ref{right-Valpha-action}, we have 
		\[
		\left( \left( f \delta_j \right) \cdot V_{\alpha_{2n}} \right) \left( t, [m] \right) = \begin{cases}
			e^{2 \pi i \left( t/\gamma - j \right)/c_{2n}} f(t) &\text{if } [m] = [j] \text{ in } \Z_{c_{2n}},\\
			0 &\text{otherwise}.
		\end{cases}
		\]
		Then for $\left( t, [m] \right) \in \R \times \Z_{c_{2n + 2}}$,
		\[
		\iota_n \left( \left( f \delta_j \right) \cdot V_{\alpha_{2n}} \right) \left( t, [m] \right) = \left( 1/\sqrt{p} \right) e^{2 \pi i \left( t/\left( p \gamma \right) - j \right)/ c_{2n}} f(t/p)
		\]
		when $m = jp + ic_0 p^{2n+1} + \ell p^{2n+2}$ for any $i \in \{ 0, \dotsc, p-1 \}$ and $\ell \in \Z$.
		Otherwise it evaluates to zero.
		
		On the other hand, since $p c_{2n} = c_0 p^{2n+1} = c_{2n+2}/p$, we have 
		\begin{align*}
			\left(\iota_{n} \left( f \delta_j \right) \cdot V_{\alpha_{2n+2}}^p \right) \left( t, [m] \right) &= e^{2 \pi i \left[ t/\gamma - m / c_{2n+2} \right]\cdot p} \iota_{n}\left( f \delta_j \right)\left( t, [m] \right)\\ 
			&= e^{2 \pi i \left( t/\gamma - m / p c_{2n} \right)} \iota_{n}\left( f \delta_j \right)\left( t, [m] \right)\\
			&= \dfrac{1}{\sqrt{p}} e^{2 \pi i \left( t/\gamma - m / p c_{2n} \right)} \sum_{i=0}^{p-1} f \delta_{jp + i c_0 p^{2n+1}} \left( \frac{t}{p}, [m] \right),
		\end{align*}
		which is nonzero for $m = jp + i c_0 p^{2n+1} + \ell p^{2n+2}$ for any $i \in \{ 0, \dotsc, p-1 \}$ and $\ell \in \Z$.
		Choose $\ell = 0$ (the computations are independent of this choice), we have that for $m =  jp + i c_0 p^{2n+1}$,
		\begin{align*}
			e^{2 \pi i \left( t/\gamma - m / p c_{2n} \right)} &= e^{2 \pi i \left[ \left( t/\gamma - jp + ic_0 p^{2n+1} \right) / p c_{2n} \right]}\\
			&= e^{ 2 \pi i \left[ \left( t/ \left( p \gamma \right)- j + ic_0 p^{2n} \right) / c_{2n} \right] }\\
			&= e^{ 2 \pi i \left[ \left( t/ \left( p \gamma \right)- j \right) / c_{2n} \right] } 
		\end{align*}
		The last equality follow from $c_0 p^{2n} = c_{2n}$.
		In conclusion, $\iota_{n} \left( \left( f \delta_j \right) \cdot V_{\alpha_{2n}} \right) = \iota_{n}\left( f \delta_j \right) \cdot V_{\alpha_{2n+2}}^p$.
	\end{proof}
	
	We now prove the following major lemma of this paper.

	\begin{lemma}\label{lemma-iota-preserves-inner-product}
		For each $n \in \N$, $\iota_n: X_{2n} \rightarrow X_{2n+2}$ preserves left inner product. That is,
		\begin{equation}
			\varphi_n\left( \inner{F_1}{F_2}_{A_{\beta_{2n}}} \right) = \inner{\iota_n\left( F_1 \right)}{\iota_n\left( F_2 \right)}_{A_{\beta_{2n + 2}}}, 
		\end{equation}
		for all $F_1, F_2 \in X_{2n}$.
	\end{lemma}
	
	We first give a formula for $\iota_{n} \left( \inner{f \delta_j}{g \delta_{j'}}_{A_{\beta_{2n}}} \right)$ in the following lemma.
	
	\begin{lemma}\label{iota-n-of-inner-product}
		Fix $f, g \in C_c\left( \R \right)$ and $j, j' \in \{ 0, \dotsc, c_{2n}-1 \}$. Let $r \in \R$ and $k = (j - j')p + \ell c_0 p^{2n+1}$ for any $\ell \in \Z$,
		\begin{align*}
			&\varphi_{n} \left( \inner{f \delta_j}{g \delta_{j'}}_{A_{\beta_{2n}}} \right)(r, k)\\
			&= \sum_{m \in \Z} f\left( c_0 p^{2n+1} r + j a_{2n} + m c_0 p^{2n} \right) \bar{g\left( c_0 p^{2n+1} r + j a_{2n} + m c_0 p^{2n} - [(j-j') + \ell c_0 p^{2n}]\gamma \right)}
		\end{align*}
		and $\iota_{n} \left( \inner{f \delta_j}{g \delta_{j'}}_{A_{\beta_{2n}}} \right)(r, k) = 0$ otherwise.
	\end{lemma}
	
	\begin{proof}
		We begin by computing $\inner{f \delta_j}{g \delta_{j'}}_{A_{\beta_{2n}}}$ for $f, g \in C_c(\R)$ and $j, j' \in \{ 0, \dotsc, c_{2n} -1 \}$.
		For any $(r, k) \in \R \times \Z$, it follow from Equation \ref{left-inner-product} that
		\begin{align*}
			\inner{f \delta_j}{g \delta_{j'}}_{A_{\beta_{2n}}} (r, k) = \sum_{m \in \Z} f\delta_j \left( c_0 p^{2n} r + m, [d_{2n} m]_{c_0 p^{2n}}  \right) \bar{g\delta_{j'}\left( c_0 p^{2n} r + m - k \gamma, [d_{2n}m - k]_{c_0 p^{2n}} \right)}
		\end{align*}
		First, the summand associated to $m$ is only nonzero when $d_{2n} m \equiv j \mod c_0p^{2n}$.
		Since $a_{2n}d_{2n} \equiv 1 \mod c_0 p^{2n}$, $d_{2n} m \equiv j \mod c_0p^{2n}$ if and only if $m \equiv j a_{2n} \mod c_0 p^{2n}$, so 
		\begin{align*}
			\begin{split}
				\inner{f \delta_j}{g \delta_{j'}}_{A_{\beta_{2n}}} (r, k) &= \sum_{m \in \Z} f \left( c_0 p^{2n} r + j a_{2n} + m c_0 p^{2n} \right) \\
				&\times \bar{g\delta_{j'}\left( c_0 p^{2n} r + j a_{2n} + m c_0 p^{2n} - k \gamma, [j - k]_{c_0 p^{2n}} \right)}.
			\end{split}
		\end{align*}
		
		Since $g \delta_{j'}$ is only nonzero when $j - k \equiv j' \mod c_0p^{2n}$, $\inner{f \delta_j}{g \delta_{j'}}_{A_{\beta_{2n}}} (r, k)$ is only nonzero when $k = (j - j') + \ell c_0 p^{2n}$ for $\ell \in \Z$, in which case 
		\begin{align*}
			\begin{split}
				\inner{f \delta_j}{g \delta_{j'}}_{A_{\beta_{2n}}} (r, k) &= \sum_{m \in \Z} f \left( c_0 p^{2n} r + j a_{2n} + m c_0 p^{2n} \right)\\
				&\times \bar{g \left( c_0 p^{2n} r + j a_{2n} + m c_0 p^{2n} - \left[ (j - j') + \ell c_0 p^{2n} \right] \gamma \right)}, 
			\end{split}
		\end{align*}
		and $\inner{f \delta_j}{g \delta_{j'}}_{A_{\beta_{2n}}} (r, k) = 0$ otherwise.
		The formula for $\varphi_{n} \left( \inner{f \delta_j}{g \delta_{j'}}_{A_{\beta_{2n}}} \right)(r, k)$ then follows from the embedding of $A_{\beta_{2n}}$ into $A_{\beta_{2n+2}}$ characterized in Lemma \ref{rotation-algebra-embedding-formula}.
	\end{proof}
	
	\begin{proof}[Proof of Lemma \ref{lemma-iota-preserves-inner-product}]
		We now compute $\inner{\iota_{n}\left( f \delta_j \right) }{ \iota_{n}\left( g \delta_{j'} \right)}_{A_{\beta_{2n+2}}}$. For a fixed pair of $j$ and $j'$ in $\{ 0, \dotsc, c_{2n} - 1 \}$ and $[m] \in \Z_{c_{2n + 2}}$, we have 
		$$
		\iota_{n}(f \delta_j)(t, [m]) = \dfrac{1}{\sqrt{p}} \sum_{i = 0}^{p-1} f\left( \dfrac{t}{p} \right) \delta_{jp + i c_0 p^{2n+1}}([m]),
		$$
		and 
		$$
		\iota_{n}(g \delta_{j'})(t, [m]) = \dfrac{1}{\sqrt{p}} \sum_{i' = 0}^{p-1} g\left( \dfrac{t}{p} \right) \delta_{j'p + i' c_0 p^{2n+1}}([m]).
		$$
		
		For $r \in \R$ and $k \in \Z$, we compute 
		\begin{align*}
			\langle \iota_{n}\left( f \delta_j \right), &\iota_{n}\left( g \delta_{j'} \right)\rangle_{A_{\beta_{2n+2}}} (r,k)\\
			&= \bigg\langle \dfrac{1}{\sqrt{p}} \sum_{i = 0}^{p-1} f\left(\frac{r}{p} \right) \delta_{jp + i c_0 p^{2n+1}}, \dfrac{1}{\sqrt{p}} \sum_{i' = 0}^{p-1} g\left(\frac{r}{p} \right) \delta_{j'p + i' c_0 p^{2n+1}} \bigg\rangle_{A_{\beta_{2n+2}}}(r,k)\\ 
			&= \dfrac{1}{p} \sum_{i = 0}^{p-1} \sum_{i' = 0}^{p-1} \big\langle f\left( \frac{r}{p} \right) \delta_{jp + i c_0 p^{2n+1}}, g\left( \frac{r}{p} \right) \delta_{j'p + i' c_0 p^{2n+1}} \big\rangle_{A_{\beta_{2n+2}}}(r, k).
		\end{align*}
		
		We focus on the inner product involving $i$ and $i'$ and compute
		\begin{align*}
			\big\langle f\left( \frac{r}{p} \right) \delta_{jp + i c_0 p^{2n+1}}, &g\left( \frac{r}{p} \right) \delta_{j'p + i' c_0 p^{2n+1}}\big\rangle_{A_{\beta_{2n+2}}}(r, k)\\ 
			= &\sum_{m' \in \Z} f \left( \dfrac{c_0 p^{2n+2} r + (jp + i c_0 p^{2n+1})a_{2n+2} + m' c_0 p^{2n+2}}{p} \right)\\
			&\times \bar{g\left( \dfrac{ c_0 p^{2n+2} r + (jp + i c_0 p^{2n+1})a_{2n+2} + m' c_0 p^{2n+2} - k \gamma}{p} \right)}\\
			= &\sum_{m' \in \Z} f \left( c_0 p^{2n+1} r + (j + i c_0 p^{2n})a_{2n+2} + m' c_0 p^{2n+1} \right)\\
			&\hspace{-1in}\times \bar{g\left( c_0 p^{2n+1} r + (j + i c_0 p^{2n})a_{2n+2} + m' c_0 p^{2n+1} - [ (j-j') + (i - i') c_0 p^{2n} + \ell c_0 p^{2n+1}] \gamma \right)}
		\end{align*}
		when $k = (j - j')p + (i - i') c_0 p^{2n+1} + \ell c_0 p^{2n+2} = (j-j')p + [(i-i') + \ell p] c_0 p^{2n+1}$ for $\ell \in \Z$, and the inner product evaluates to zero otherwise.
		
		Since $i$ and $i'$ are both integers between $0$ and $p-1$, $(i-i') + \ell p$ must be $K + \ell p$ for some $K \in \{ 0, \dotsc, p-1 \}$.
		Specifically, there are $p - K$ pairs of $(i, i')$, with $i \geq i'$, such that $(i-i') + \ell p = K + \ell p$, and $K$ pairs of $(i, i')$, with $i < i'$, such that $(i - i') + (\ell + 1)p = K + \ell p$.
		Together, there are exactly $p$ nonzero inner products contributing to the value of $\inner{\iota_{n}(f\delta_j)}{ \iota_{n}(g \delta_{j'})}_{A_{\beta_{2n+2}}} (r, k)$ for $k =(j-j')p + [K + \ell p] c_0 p^{2n+1}$.
		Therefore, for each fixed $i$, there are $p$ nonzero inner products, each with value 
		
		\begin{align*}
			&\sum_{m' \in \Z} f \left( c_0 p^{2n+1} r + (j + i c_0 p^{2n})a_{2n+2} + m' c_0 p^{2n+1} \right)\\
			&\hspace{2cm} \times \bar{g\left( c_0 p^{2n+1} r + (j + i c_0 p^{2n})a_{2n+2} + m' c_0 p^{2n+1} - [ (j-j') + \left[ K + \ell p \right] c_0 p^{2n}] \gamma \right)}.
		\end{align*}
		
		Therefore, 
		
		\begin{align*}
			&\inner{\iota_{n}\left( f \delta_j \right)}{\iota_{n}\left( g \delta_{j'} \right)}_{A_{\beta_{2n+2}}} (r, (j-j')p + \left[ K + \ell p \right] c_0 p^{2n+1} )\\
			= &\dfrac{1}{p} \cdot p \sum_{i = 0}^{p-1} \Big( \sum_{m' \in \Z} f \left( c_0 p^{2n+1} r + (j + i c_0 p^{2n})a_{2n+2} + m' c_0 p^{2n+1} \right)\\
			&\times \bar{g\left( c_0 p^{2n+1} r + (j + i c_0 p^{2n})a_{2n+2} + m' c_0 p^{2n+1} - [ (j-j') + \left[ K + \ell p \right] c_0 p^{2n}] \gamma \right)} \Big)\\ 
			= &\sum_{i = 0}^{p-1} \Big( \sum_{m' \in \Z} f \left( c_0 p^{2n+1} r + j a_{2n+2} + \left( i a_{2n+2}  + m' p \right) c_0 p^{2n} \right)\\
			&\times \bar{g\left( c_0 p^{2n+1} r +  j a_{2n+2} + \left( i a_{2n+2}  + m' p \right) c_0 p^{2n} - [ (j-j') + \left[ K + \ell p \right] c_0 p^{2n}] \gamma \right)} \Big)
		\end{align*}

		Since $a_{2n+2}$ is relatively prime to $p$, as it is relatively prime to $c_{2n} = c_0p^{2n}$ and $i a_{2n+2}$ goes through the modular classes of $p$ exactly once as $i$ goes from $0$ to $p-1$.
		Therefore,
		\[
		\bigcup_{i = 0}^{p -1} \bigcup_{m' \in \Z} i a_{2n+2} + m'p = \Z,
		\]
		so 
		\begin{align*}
			&\inner{\iota_{n}\left( f \delta_j \right)}{\iota_{n}\left( g \delta_{j'} \right)}_{A_{\beta_{2n+2}}} (r, (j-j')p + \left[ K + \ell p \right] c_0 p^{2n+1} )\\
			= &\sum_{m \in \Z} f \left( c_0 p^{2n+1} r + j a_{2n+2} + m c_0 p^{2n} \right)\\
			&\times \bar{g\left( c_0 p^{2n+1} r +  j a_{2n+2} + m c_0 p^{2n} - [ (j-j') + \left[ K + \ell p \right] c_0 p^{2n}] \gamma \right)}.
		\end{align*}
		
		We now need to relate $a_{2n+2}$ and $a_{2n}$.
		In the proof of Proposition \ref{construction-beta}, we showed that 
		\begin{equation*}
			\dfrac{a_{2n+2} - a_{2n}}{c_{2n}} = z \in \Z,
		\end{equation*}
		so $a_{2n+2} = a_{2n} + z c_0 p^{2n}$ and we have 
		\begin{align*}
			&\inner{\iota_{n}\left( f \delta_j \right)}{\iota_{n}\left( g \delta_{j'} \right)}_{A_{\beta_{2n+2}}} (r, (j-j')p + \left[ K + \ell p \right] c_0 p^{2n+1} )\\
			= &\sum_{m \in \Z} f \left( c_0 p^{2n+1} r + j \left( a_{2n} + z c_0 p^{2n} \right) + m c_0 p^{2n} \right)\\
			&\times \bar{g\left( c_0 p^{2n+1} r +  j \left( a_{2n} + z c_0 p^{2n} \right) + m c_0 p^{2n} - [ (j-j') + \left[ K + \ell p \right] c_0 p^{2n}] \gamma \right)}\\
			= &\sum_{m \in \Z} f \left( c_0 p^{2n+1} r + j a_{2n} + \left( j z + m \right) c_0 p^{2n} \right)\\
			&\times \bar{g\left( c_0 p^{2n+1} r +  j a_{2n} + \left( j z + m \right) c_0 p^{2n} - [ (j-j') + \left[ K + \ell p \right] c_0 p^{2n}] \gamma \right)}\\
			= &\sum_{m \in \Z} f \left( c_0 p^{2n+1} r + j a_{2n} + m c_0 p^{2n} \right)\\
			&\times \bar{g\left( c_0 p^{2n+1} r +  j a_{2n} + m c_0 p^{2n} - [ (j-j') + \left[ K + \ell p \right] c_0 p^{2n}] \gamma \right)}
		\end{align*}
		
		Since $j$ and $z$ are fixed, $jz + m$ goes through the integers exactly once as $m$ runs through $\Z$, so the last equality follows.
		
		Lastly, as $K$ goes from $0$ to $p-1$, and $\ell$ goes over all integers, the sum $K + \ell p$ will go over the integers exactly once, so we have that, for $k = (j - j')p + \ell c_0 p^{2n+1}$,
		\begin{align*}
			&\inner{\iota_{n}\left( f \delta_j \right)}{\iota_{n}\left( g \delta_{j'} \right)}_{A_{\beta_{2n+2}}} (r, k)\\
			&= \sum_{m \in \Z} f \left( c_0 p^{2n+1} r +  j a_{2n} + m c_0 p^{2n} \right) \bar{g\left(c_0 p^{2n+1} r +  j a_{2n} + m c_0 p^{2n} - [ (j-j') + \ell c_0 p^{2n} ] \gamma \right)} 
		\end{align*}
		and $\inner{\iota_{n}\left( f \delta_j \right) }{ \iota_{n}\left( g \delta_{j'} \right)}_{A_{\beta_{2n+2}}} (r, k) = 0$ otherwise, which is exactly $\varphi_{n} \left( \inner{f\delta_j}{g\delta_{j'}}_{A_{\beta_{2n}}} \right)(r, k)$ in Lemma \ref{iota-n-of-inner-product}.
		We conclude that
		\[
		\varphi_{n} \left( \inner{f \delta_j}{g \delta_{j'}}_{A_{\beta_{2n}}} \right) = \inner{\iota_{n}\left( f \delta_j \right)}{\iota_{n}\left( g \delta_{j'} \right)}_{A_{\beta_{2n+2}}}.
		\]
	\end{proof}
	
	The main result of this section is the following.
	\begin{theorem}\label{main-theorem}
		Fix prime $p$ and let $\alpha = \left( \alpha_n \right)_{n \in \N} \in \Xi_p$ be an irrational sequence with $p \alpha_{n+1} = \alpha_n + x_n$ for all $n \in \N$.
		Let $P$ be a projection in $M_m\left( A_{\alpha_0} \right)$, $m \geq 1$, that satisfies Condition \ref{the-condition}.
		Then $\left( X_{2n}, \iota_n \right)_{n \in \N}$ is a directed system of equivalence bimodules adapted to the sequence $\left( P M_m\left( A_{\alpha_{2n}} \right)P, \varphi_n \right)_{n \in \N} = \left( A_{\beta_{2n}}, \varphi_n \right)_{n \in \N}$ and $\left( A_{\alpha_{2n}}, \psi_n \right)_{n \in \N}$.
	\end{theorem}
	
	\begin{proof}
		We showed in Proposition \ref{construction-beta} that $X_{2n}$ is an $A_{\beta_{2n}}$-$A_{\alpha_{2n}}$-equivalence bimodule for each $n$.
		By Lemma \ref{iota-preserves-left-action} and \ref{lemma-iota-preserves-inner-product}, each embedding preserves left module action and left inner product, that is, 
		\[
		\iota_{n} \left( f \cdot F \right) = \varphi_n\left( f \right) \cdot \iota_n \left( F \right), \text{ and } \varphi_n\left( \inner{F_1}{F_2}_{A_{\beta_{2n}}} \right) = \inner{\iota_n\left( F_1 \right)}{\iota_n\left( F_2 \right)}_{A_{\beta_{2n + 2}}}, 
		\]
		for $f \in A_{\beta_{2n}}$ and $F$, $F_1$, and $F_2$ in $X_{2n}$.
		Moreover, Lemma \ref{iota-preserves-right-action} shows that each embedding respects the right modules action.
		
		It only remains to show that $\psi_n(\langle F_1, F_2 \rangle_{A_{\alpha_{2n}}})=  \langle \iota_n(F_1), \iota_n(F_2) \rangle_{A_{\alpha_{2n+2}}}$ for all $F_1, F_2 \in X_{2n}$.
		First, observe that for any $F, G, H \in X_{2n}$, we have 
		\[
		\langle F, G \rangle_{A_{\beta_{2n}}} \cdot H = F \cdot \langle G, H \rangle_{A_{\alpha_{2n}}},
		\]
		so
		\[
		\iota_n\left( \langle F, G \rangle_{A_{\beta_{2n}}} \cdot H \right) = \iota_n \left( F \cdot \langle G, H \rangle_{A_{\alpha_{2n}}} \right),
		\]
		or 
		\[
		\varphi_n \left( \langle F, G \rangle_{A_{\beta_{2n}}} \right) \cdot \iota_n\left( H \right) = \iota_n \left( F \right) \cdot \psi_n \left( \langle G, H \rangle_{A_{\alpha_{2n}}} \right).
		\]
		By Lemma \ref{lemma-iota-preserves-inner-product}, the left side of the equality is equal to $\langle \iota_n\left( F \right), \iota_n\left( G \right) \rangle_{A_{\beta_{2n+2}}} \cdot \iota_n\left( H \right)$, which is equal to $\iota_n\left( F \right) \cdot \langle \iota_n\left( G \right), \iota_n\left( H \right) \rangle_{A_{\alpha_{2n+2}}}$.
		Therefore, for any fixed $G, H \in X_{2n}$,
		\[
		\iota_n\left( F \right) \cdot \langle \iota_n\left( G \right), \iota_n\left( H \right) \rangle_{A_{\alpha_{2n+2}}} = \iota_n \left( F \right) \cdot \psi_n \left( \langle G, H \rangle_{A_{\alpha_{2n}}} \right)
		\]
		for all $F \in X_{2n}$.
		This is sufficient for us to conclude that
		\[
		\langle \iota_n\left( G \right), \iota_n\left( H \right) \rangle_{A_{\alpha_{2n+2}}} = \psi_n \left( \langle G, H \rangle_{A_{\alpha_{2n}}} \right)
		\]
		for all $G, H \in X_{2n}$ as follows.
		We have shown in Lemma \ref{lemma-iota-preserves-inner-product} that $\langle \iota_n \left( X_{2n} \right), \iota_n \left( X_{2n} \right) \rangle_{A_{\beta_{2n + 2}}} = \psi_n\left(A_{\beta_{2n}} \right)$.
		One checks that $\langle \iota_n \left( X_{2n} \right), \iota_n \left( X_{2n} \right) \rangle_{A_{\alpha_{2n + 2}}} \subseteq \psi_n\left(A_{\alpha_{2n}} \right)$.
		Therefore, we can regard $\iota_n\left( X_{2n} \right)$ as a $A_{\beta_{2n}}$-$A_{\alpha_{2n}}$-equivalence bimodule.
		By \cite[Proposition 3.8]{RW98}, the map $\phi : A_{\alpha_{2n}} \rightarrow \mathcal{K}\left( \iota_n \left( X_{2n} \right) \right)$ from $A_{\alpha_{2n}}$ to the compact operators on $\iota_n\left(X_{2n} \right)$ given by $a \mapsto \phi_a(x) = x \cdot a$ is an isomorphism of C*-algebras.
		Let $a = \langle \iota_n\left( G \right), \iota_n\left( H \right) \rangle_{A_{\alpha_{2n+2}}}$ and $a' = \psi_n \left( \langle G, H \rangle_{A_{\alpha_{2n}}} \right)$, we have $\phi_a = \phi_{a'}$ on all of $\iota_n\left( X_{2n} \right)$, which implies that $a = a'$, as we desired to show.
	\end{proof}	
	
	\begin{corollary}
		Let $\A_{\alpha}^\S = \varinjlim A_{\alpha_{2n}}$ and $\A_{\beta}^\S = \varinjlim A_{\beta_{2n}}$ be given as in Theorem \ref{main-theorem}.
		Then $\A_{\alpha}^\S$ and $\A_{\beta}^\S$ are Morita equivalent. 
	\end{corollary}
	
	\begin{proof}
		This follows directly from \cite[Theorem 3.4]{LP3} by taking the direct limit of $\left(X_{2n}, \iota_n \right)_{n \in \N}$.
	\end{proof}

	\subsection{Relating back to the Heisenberg bimodules.} We now state how the two constructions of equivalence bimodules between noncommutative solenoids are related.
	Fix a noncommutative solenoid $\A_\alpha^\S$ with $\alpha = \left( \alpha_n \right)_{n \in \N} \in \Xi_p$ and $p \alpha_{n+1} = \alpha_{n} + x_n$, $x_n \in \{ 0, \dots, p-1 \}$ for all $n$. Assume also that $\alpha_0 \neq 0$ and $x_0 \neq 0$, so the associated $p$-adic integer $x_\alpha = \sum_{j = 0}^\infty x_j p^j$ has inverse $x_\alpha^{-1} = \sum_{j = 0}^\infty y_j p^j \in \Z_p$.
	
	For such a fixed $\A_\alpha^\S$, the Morita equivalent solenoid $\A_\beta^
	\S$ that arises from the Heisenberg bimodules construction (see Theorem \ref{beta-formula-nonzero}) is given by
	\[
	\beta = \left( \dfrac{1}{\alpha_0 p^n} + \dfrac{\sum_{j = 0}^{n-1} y_j p^j}{p^n} \right)_{n \in \N}.
	\]
	
	Following Remark \ref{condition-special-case} and Notation \ref{nc-solenoid-notation}, the projection $P \in A_{\alpha_0}$ with $\tau_{\alpha_0}\left(P \right) = \alpha_0$ satisfies Condition \ref{the-condition}, so by Proposition \ref{construction-beta}, the noncommutative solenoid $\A_\beta^\S = \varinjlim P A_{\alpha_{2n}} P$ is determined by
	\[
	\beta_{2n} = \beta_{2n} = \dfrac{a_{2n} \alpha_{2n} + b_{2n}}{c_{2n} \alpha_{2n} +d_{2n} },
	\]
	with
	\[
	a_{2n} = \sum_{j=0}^{2n-1} y_j p^j,\ b_{2n} = \left( \left(  \sum_{ j = 0}^{2n-1} y_j p^j  \right)\left( - \sum_{ j = 0}^{2n-1} x_j p^j \right) + 1 \right) p^{-2n},
	\]
	\[
	c_{2n} = p^{2n},\ \text{and } d_{2n} =  - \sum_{j=0}^{2n-1} x_j p^j.
	\]
	By Lemma \ref{inverse-multiply-one}, $b_2n$ is an integer and one checks that $a_{2n}d_{2n} - b_{2n}c_{2n} = 1$ for all $n$.
	Observe that, since $p^n \alpha_n = \alpha_0 + \sum_{j = 0}^{n-1} x_j p^j$, 
	\begin{align*}
		\beta_{2n} &= \dfrac{a_{2n} \alpha_{2n} + b_{2n}  }{ p^{2n}\alpha_{2n} + \left( - \sum_{j=0}^{2n-1} x_j p^j \right) }\\
		&= \dfrac{ \left( \sum_{ j = 0}^{2n-1} y_j p^j \right) \alpha_{2n} + \left( \left(  \sum_{j = 0}^{2n-1} y_j p^j  \right)\left( - \sum_{ j = 0}^{2n-1} x_j p^j \right) + 1 \right) p^{-2n}  }{ p^{2n}\alpha_{2n} + \left( - \sum_{j=0}^{2n-1} x_j p^j\right) }\\
		&= \dfrac{ \left( \sum_{ j = 0}^{2n-1} y_j p^j  \right)\alpha_{2n}p^{2n} + \left( \left(  \sum_{j = 0}^{2n-1} y_j p^j  \right)\left( - \sum_{ j = 0}^{2n-1} x_j p^j \right) + 1 \right) }{\alpha_{0}p^{2n}}\\
		&= \dfrac{ \left( \sum_{ j = 0}^{2n-1} y_j p^j \right) \left( \alpha_{0} + \sum_{j = 0}^{2n-1} x_j p^j \right) + \left( \left(  \sum_{j = 0}^{2n-1} y_j p^j  \right)\left( - \sum_{ j = 0}^{2n-1} x_j p^j \right) + 1 \right) }{\alpha_{0}p^{2n}}\\
		&= \dfrac{ \alpha_{0} \left( \sum_{ j = 0}^{2n-1} y_j p^j \right) + 1 }{\alpha_{0}p^{2n}}\\
		&= \dfrac{1}{\alpha_{0} p^{2n}} + \dfrac{\sum_{ j = 0}^{2n-1} y_j p^j}{p^{2n}},
	\end{align*}
	which are exactly the even entries of $\beta$ calculated from the Heisenberg bimodules construction.
	By Remark \ref{determine-alpha}, since the two $\beta$'s agree at infinitely many entries, they must determine the same noncommutative solenoid.
	
	We summarize the above discussion as the following result.
	
	\begin{proposition}\label{relate-constructions}
		 Fix a noncommutative solenoid $\A_\alpha^\S$ with $\alpha = \left( \alpha_n \right)_{n \in \N} \in \Xi_p$ and $p \alpha_{n+1} = \alpha_{n} + x_n$, $x_n \in \{ 0, \dots, p-1 \}$ for all $n$. Assume also that $\alpha_0 \neq 0$ and $x_0 \neq 0$.
		 The Morita equivalent solenoid $\A_\beta^\S$ formed from the directed system of equivalence bimodule construction, using projection $P \in A_{\alpha_0}$ with $\tau_{\alpha_0}\left( P \right) = \alpha_0$, is the same as the one from the Heisenberg bimodule construction.
	\end{proposition}
	
	\begin{remark}
		It remains an open question as to whether or not we can extend this result to the $\A_\beta^\S$ given in Theorem \ref{beta-formula-general}, which generalizes Theorem \ref{beta-formula-nonzero} used in the above discussion. 
	\end{remark}

	\section{Morita equivalence of irrational noncommutative solenoids}\label{S-morita-problem}
	
	We are now ready to prove our main result, which addresses the Morita equivalence problem for irrational noncommutative solenoids. We start with the following observation.
	
	\begin{lemma}\label{no-matrix-over-irr-rotation-algebra}
		Let $\A_\alpha^\S$ be an irrational noncommutative solenoid, then it does not contain a unital C*-subalgebra isomorphic to $M_m\left( A_\theta \right)$ for any $m$ strictly greater than $1$.
	\end{lemma}
	
	\begin{proof}
		In \cite{LP2}, Latr{\'e}moli{\`e}re and Packer showed that an irrational noncommutative solenoid is simple and has a unique tracial state, which we denote by $\tau$.
		Moreover, the range of this trace on $K_0\left( \A_\alpha^\S \right)$ is given by $\{ z + y \alpha_k \:: z, y \in \Z, k \in \N \}$.
		The restriction of $\tau$ to any unital C*-subalgebra B of $\A_\alpha^\S$ is again a tracial state on $B$.
		For a contradiction, suppose that $B \cong M_m\left( A_\theta \right)$ for $m \geq 2$, then the restriction of $\tau$ must be the unique tracial state on $M_m\left( A_\theta \right)$.
		Since $\diag\left( 1, 0, \dotsc, 0 \right) \in M_m\left( A_\theta \right)$ is a projection with $\tau\left( \diag\left( 1, 0, \dotsc, 0 \right) \right) = 1/m$, the range $\tau$ on $K_0\left( \A_\alpha^\S \right)$ must contain the fraction $1/m$ and we arrive at a contradiction. 
	\end{proof}
	
	It follows from the lemma below that every projection in $M_m\left( \A_\alpha^\S \right)$ is unitarily equivalent to a projection $P \in M_m\left( A_{\alpha_{2k}} \right)$ for some $k \in \N$.
	
	\begin{lemma}\label{unitarily-equivalent-projections}
		Let $A$ be the direct limit of the direct sequence of C*-algebras $\left( A_n, \varphi_n \right)_{n = 1}^\infty$ and suppose that $\varphi^n : A_n \rightarrow A$ is the natural map for each $n$. If $p$ is a projection in $A$, then there is a $k \geq 1$ and a projection $q \in A_k$ such that $p$ and $\varphi^k\left( q \right)$ are unitarily equivalent.
	\end{lemma}
	
	\begin{proof}
		By \cite[Proposition 5.2.6]{WO93}, it is sufficient to find a projection $q \in A_k$ such that $\norm{p - \varphi^k\left( q \right)}$ is strictly less than $1$.
		Such a projection exists by \cite[Corollary 5.1.7]{WO93}.
	\end{proof}
	
	\begin{theorem}\label{main-result}
		Let $p$, $q$ be prime numbers. Let $\alpha \in \Xi_p$ and $\beta \in \Xi_q$. Then the following statements are equivalent.
		\begin{enumerate}[(1)]
			\item $\A_\alpha^\S$ and $\A_\beta^\S$ are Morita equivalent.
			\item $p = q$ and there exists a projection $P \in M_m\left( A_{\alpha_{2k}} \right)$ for some $k, m \in \N$ that satisfies Condition \ref{the-condition}, and $\A_\beta^\S \cong \varinjlim_{n \geq k} P M_m\left( A_{\alpha_{2n}} \right) P$.
		\end{enumerate}
	\end{theorem}
	
	\begin{remark}
		It is clear that truncating the first $k$ terms of $\alpha = \left( \alpha_n \right)_{n \in \N}$ does not change the resulting noncommutative solenoid.
		For this reason, we say that a projection $P \in M_m\left( A_{\alpha_{2k}}\right)$ satisfies Condition \ref{the-condition} to mean: $\tau_{\alpha_{2n}}\left( P \right) = a_{2n} \alpha_{2n} + b_{2n}$ and $\gcd\left( pc_{2n}, d_{2n} - c_{2n} x_{2n} \right) = 1$.
		Essentially, our results in the previous section apply to any projection in $M_m\left( A_{\alpha_{k}} \right)$ for any $k \in \N$ by truncating the first $k$ terms of $\alpha$ starting from zero.
		Without loss of generality, we can always assume $k$ is even.
	\end{remark}
	
	\begin{proof}[Proof of Theorem \ref{main-result}]
		$(2)$ implying $(1)$ follows from the directed system of equivalence bimodules construction, as in Theorem \ref{main-theorem} and its Corollary, with projection $P$ and truncated sequence $\left( A_{\alpha_{2n}}, \varphi_n \right)_{n \geq k}$.
		
		We show that $(1)$ implies $(2)$ as follows.
		In \cite[Theorem 3.7]{LP1}, the $K$-theory of noncommutative solenoids is computed: in particular, for $\alpha \in \Xi_p$,
		\[
		K_1\left( \A_\alpha^\S \right) = \Z\left[ \dfrac{1}{p} \right] \times \Z\left[ \dfrac{1}{p} \right].
		\]
		It is clear that for $p$ and $q$ primes, $\Z\left[ 1/p \right] \times \Z\left[ 1/p \right] \cong \Z\left[1/q \right] \times \Z\left[1/q \right]$ if and only if $p = q$.
		Since Morita equivalent C*-algebras have isomorphic $K_1$ groups, we must have $p = q$. 
		
		Since both $\A_\alpha^\S$ and $\A_\beta^\S$ are unital, $\A_\beta^\S \cong P M_m\left( \A_\alpha^\S \right) P$ for a suitable $m \geq 1$ and a projection $P \in M_m\left( \A_\alpha^\S \right)$ \cite[Proposition 2.1]{Rieffel81}.
		Recall that $\A_\alpha^\S$ is the direct limit of the sequence of increasing irrational rotation algebras associated to the $\alpha_{2n}$'s.
		By Lemma \ref{unitarily-equivalent-projections}, we can then assume that $P \in M_m\left( A_{\alpha_{2k}} \right)$ for a large enough $k$.
		It is evident that
		\[
		\A_\beta^\S \cong P M_m\left( \A_\alpha^\S \right) P \cong \varinjlim_{n \geq k} P M_m\left( A_{\alpha_{2n}} \right) P.  
		\]
		It remains to show that $P$ must satisfy Condition \ref{the-condition}.
		Assume, for the sake of contradiction, that $P$ does not satisfy the Condition.
		Then by Lemma \ref{condition-useful-form}, there exists $N \geq k$ such that $\tau_{\alpha_{2N}} \left(P \right) = \ell \left( c\alpha_{2N} + d \right)$, with $\gcd\left( c, d \right) = 1$.
		As discussed in the paragraph following Lemma \ref{iso-right-module}, $P M_m\left( A_{\alpha_{2n}} \right) P$ is then isomorphic to an $\ell \times \ell$ matrix algebra over some irrational rotation algebra, say $A_\beta$.
		Moreover, it is easy to check that $\ell$ is a common divisor of $c_{2n}$ and $d_{2n}$ for all $n \geq N$.
		So the direct limit C*-algebra contains $M_\ell\left( A_{\beta} \right)$ as a subalgebra.
		This is a contradiction to Lemma \ref{no-matrix-over-irr-rotation-algebra}.
	\end{proof}
	
	We end this section with an example where Condition \ref{the-condition} is not satisfied. In particular, we see that the direct limit of $P A_{\alpha_{2n}} P$ can be a direct limit C*-algebra of matrix algebras that are strictly increasing in size.
	
	\begin{example}
		For $\alpha_0 \neq 0$, consider the irrational noncommutative solenoid $\A_\alpha^\S$ given by 
		\[
		\alpha = \left( \alpha_0, \dfrac{\alpha_0}{p}, \dfrac{\alpha_0}{p^2}, \dotsc \right) \in \Xi_p.
		\]
		Let $P$ be the projection in $A_{\alpha_0}$ with trace $\alpha_0$.
		As a projection in $A_{\alpha_{2n}}, \tau_{\alpha_{2n}} \left( p \right) = p^{2n} \alpha_{2n}$.
		By Lemma \ref{iso-right-module}, $P A_{\alpha_{2n}} P \cong M_{p^{2n}}\left( A_{\frac{1}{\alpha_{2n}}} \right)$.
		Therefore, by Lemma \ref{no-matrix-over-irr-rotation-algebra}, $PM_n(\A_{\alpha})P$ cannot be a noncommutative solenoid.
		However, it would be interesting to further investigate the structure of the C*-algebras $\varinjlim P A_{\alpha_{2n}} P \cong \varinjlim M_{p^{2n}}\left( A_{\frac{1}{\alpha_{2n}}} \right)$, and how they are related to the noncommutative solenoids. 
	\end{example}
	
	\section*{Acknowledgments}
	
	This work makes up part of the author's Ph.D. thesis for the University of Colorado Boulder.
	The author would like to thank his thesis advisor Dr. Judith Packer. He is beyond grateful for her continued support and guidance.
	This work would not have come to fruition without her.


\bibliographystyle{plain}	
\bibliography{refs}		
\end{document}